\documentclass{article}
\usepackage{amsmath,amssymb,amsthm}
\usepackage{graphicx,subfigure,float,url,color}
\usepackage{xcolor}
\usepackage{mathrsfs}
\usepackage[colorlinks=true]{hyperref}
\usepackage{pdfsync}

\newtheorem{theorem}{Theorem}[section]
\newtheorem{proposition}[theorem]{Proposition}
\newtheorem{lemma}[theorem]{Lemma}
\newtheorem{corollary}[theorem]{Corollary}
\theoremstyle{definition}
\newtheorem{definition}[theorem]{Definition}
\newtheorem{remark}[theorem]{Remark}

\renewcommand{\geq}{\geqslant}
\renewcommand{\leq}{\leqslant}

\newcommand{\R}{\mathbb{R}}
\newcommand{\N}{\mathbb{N}}
\newcommand{\vol}{\mathrm{vol}}

\newcommand{\bigO}{\mathrm{O}}

\newcommand{\bm}{\beta_{\mathrm{max}}}

\DeclareMathOperator{\dist}{dist}

\title{Weyl asymptotics for singular metrics with a variable boundary degeneracy exponent}
\author{Yves Colin de Verdi\`ere\footnote{Universit\'e Grenoble-Alpes, Institut Fourier, Unit{\'e} mixte de recherche CNRS-UGA 5582, BP 74, 38402-Saint Martin d'H\`eres Cedex (France) (\texttt{yves.colin-de-verdiere@univ-grenoble-alpes.fr})},
\and
Charlotte Dietze\footnote{Sorbonne Universit\'e, CNRS, Universit\'e Paris Cit\'e, Laboratoire Jacques-Louis Lions (LJLL), F-75005 Paris, France (\texttt{charlotte.dietze@sorbonne-universite.fr})}, 
\and
Emmanuel Tr\'elat\footnote{{\bf Corresponding author.} Sorbonne Universit\'e, CNRS, Universit\'e Paris Cit\'e, Inria, Laboratoire Jacques-Louis Lions (LJLL), F-75005 Paris, France (\texttt{emmanuel.trelat@sorbonne-universite.fr})}
}
\date{}

\begin{document}
\maketitle

\begin{abstract}
We consider a compact smooth manifold $X$ of dimension $n+1$ with boundary $M=\partial X$. In a collar neighborhood of $M$, we assume that the metric has the form $g=u^{-\alpha}\bar g$, where $u$ is a boundary defining function, $\alpha\in C^1(M;[0,2))$ and $\bar g$ is a $C^1$ Riemannian metric up to $M$. Since $\alpha<2$, the boundary lies at finite $g$-distance and $(X,g)$ is a singular metric space. We study the Weyl asymptotics of the Friedrichs Laplacian $\triangle_g$ when the degeneracy exponent $\alpha$ varies along $M$.

If the maximum $\alpha_{\mathrm{max}}$ of $\alpha$ on $M$ is strictly larger than the critical value $\alpha_c=\frac{2}{n+1}$, then we prove that the points where $\alpha$ is close to $\alpha_{\mathrm{max}}$ govern the leading term in the Weyl asymptotics. If $\alpha_{\mathrm{max}}\leq\alpha_c$, then the leading term is governed by the truncated volume $\vol_g(\{\dist(\cdot,M)>\lambda^{-1/2}\})$. When the maximum set of $\alpha$ is Morse-Bott, we compute the associated constants and the logarithmic corrections. To the best of our knowledge, this is the first Weyl law in this setting with a boundary-dependent degeneracy exponent. The results highlight a sharp transition at $\alpha_c$ between a boundary-dominated non-classical regime and a truncated-volume regime.
\end{abstract} 


\medskip



\section{Introduction and main results}
\subsection{Setting}

Let $X$ be a smooth compact manifold of dimension $n+1$ with boundary $M=\partial X$ of dimension $n$ and let $g$ be a Riemannian metric on $X\setminus M $ such that, in a collar neighborhood  $(0,1]_u\times M$ of the boundary,
\begin{equation*}
g = u^{-\alpha} \bar{g} , 
\end{equation*}
where $\alpha:M\to [0,2)$ is of class $C^1$ and $\bar g$ is a Riemannian metric on $X$ of class $C^1$ up to the boundary.
We set
$$
\beta = \frac{2\alpha}{2-\alpha}
\qquad\textrm{and}\qquad
\bm = \sup_{m\in M}\beta(m) .
$$
The assumption $\alpha<2$ implies that every point of $X$ lies at finite $g$-distance from the boundary.
The case $\alpha=0$ corresponds to an ordinary Riemannian metric on $X$.

We consider the (nonnegative) Laplace-Beltrami operator $\triangle_g$ associated with $g$, realized by the Friedrichs extension on $L^2(X,\vol_g)$ with core $C_0^\infty (X\setminus \partial X)$. 
It is known, see for instance \cite{CDHDT24}, that $\triangle_g$ has discrete spectrum with eigenvalues $0<\lambda_1\leq\cdots\leq\lambda_j\leq\cdots$. 
This holds for a variable exponent as well: the Hardy-type lower bound of Lemma \ref{le:0ev}, whose proof does not use the constancy of $\beta$, shows that the bottom of the spectrum near the boundary tends to $+\infty$ at the wavelength scale, so that the form domain embeds compactly into $L^2(X)$ and $\triangle_g$ has compact resolvent.
 We study the asymptotic behavior of the counting function $N(\lambda)=\#\{j\in\N^*\,\mid\,\lambda_j\leq \lambda\}$ as $\lambda\to +\infty$.

For a \emph{fixed constant} exponent $\alpha>0$, the spectral asymptotics of $\triangle_g$ is known (see \cite{CDHDT24}) and exhibits three regimes separated by the critical value $\alpha_c=2/(n+1)$ or $\beta_c=2/n$. For notational convenience, we mostly work with the parameter $\beta  = 2\alpha/(2-\alpha)$.
The constant-exponent result of \cite{CDHDT24} reads as follows.

\begin{theorem}[\cite{CDHDT24}]\label{theo:main0}
Assume that $\alpha \in [0,2)$ is constant and set $\beta  =2\alpha/(2-\alpha)$. Let $h_0$ be the metric on $M$ that is the restriction of $\bar{g}$ to the boundary with a suitable choice of a transverse coordinate $u$ (see Section \ref{sec:normal} for the precise definition of $h_0$). 
Let $\triangle_g$ be the Friedrichs extension of the Laplace-Beltrami operator on $X$ with respect to $g$, with core $C_0^\infty(X\setminus \partial X)$. Let $\gamma_n=b_n/(2\pi)^n$ denote the semiclassical constant, where $b_n$ is the volume of the unit ball in $\R^n$, and let $N(\lambda)$ be the counting function of $\triangle_g$.
\begin{enumerate}
\item If $\beta > \beta_c$ then $N(\lambda ) \sim  A(\beta, n)\,\vol_{h_0}(M)\, \lambda^{d_H/2}$ as $\lambda \rightarrow +\infty$, where $d_H = n(1+\beta/2) $ is the Hausdorff dimension of $(X,g)$ and $A(\beta, n) = \gamma_n \, \zeta_{\beta,n} (d_H/2)$, with $\zeta_{\beta,n} $ the spectral $\zeta$ function of the 1D operator $P_1 = -\partial_x^2 + \frac{\beta n (\beta n + 4 )}{16 x^2} + x^\beta$ on $L^2 (\R^+)$ with Dirichlet boundary conditions. 
\item If $\beta = \beta_c = 2/n$ then $N(\lambda ) \sim  \frac{1}{2}\gamma_{n+1} \, \vol_{h_0}(M) \lambda ^{(n+1)/2}\ln \lambda$ as $\lambda \rightarrow +\infty$.
\item If $\beta < \beta_c$ then $N(\lambda ) \sim  \gamma_{n+1} \, \vol_g(X) \, \lambda ^{(n+1)/2}$ as $\lambda \rightarrow +\infty$.
\end{enumerate}
\end{theorem}

\begin{remark}\label{rem:regularity-cdhdt}
In our setting, the Minkowski dimension and the Hausdorff dimension agree.
The statement of Theorem \ref{theo:main0} is proved in \cite{CDHDT24} under smoothness assumptions. For our purposes, only the leading terms are used, and they are stable under $C^1$ perturbations and, more generally, under uniform quasi-isometries. In particular, one may smooth the coefficients and transfer the leading asymptotics back to the original metric using Lemmas \ref{lemm:quasi} and \ref{lem:counting-quasi} in Appendix \ref{app:quasi}.
\end{remark}

\subsection{Main results}
We extend this constant-exponent picture to the genuinely variable case. A natural guess is that the asymptotics should be governed by the local Hausdorff dimension $d_H(m)=n\big(1+\frac{\beta(m)}{2}\big)$. However, the constant-exponent asymptotics is not uniform as $\beta\downarrow\beta_c$: the coefficient $A(\beta,n)$ blows up as $\beta\downarrow\beta_c$ (see Section \ref{sec:phase+}). Likewise, in the subcritical regime the coefficient $\gamma_{n+1}\vol_g(X)$ is not uniform as $\beta\uparrow\beta_c$, since the volume diverges at the threshold, already in the model $dx^2+x^{-\beta}h_0$, where $\vol_{g_\beta}((0,1)\times M)=\frac{1}{1-\frac{n\beta}{2}}\vol_{h_0}(M)$. For a variable $\beta(m)$, these two singular behaviors prevent a direct uniform reduction to the constant-exponent case near points where $\beta(m)$ is close to $\beta_c$.

This leads naturally to two regimes. If $\bm>\beta_c$, the leading term comes from the supercritical region $\{\beta>\beta_c\}$ and is obtained by localizing along $M$ and freezing $\beta$ on mesoscopic cells (Theorem \ref{thm:main}, item 1). If $\bm\leq\beta_c$, the relevant scale is the truncated volume, and one uses the Weyl law expressed in terms of $\vol_g(\{\dist(\cdot,M)>\lambda^{-1/2}\})$ (Theorem \ref{thm:main}, item 2), in the spirit of \cite{CPR24}.

\medskip

Recall that $\beta$ and $\bar g$ are of class $C^1$.
We denote by $d\vol_{h_0}(m)$ the volume form on $M$ associated with the metric $h_0$ defined in Section \ref{sec:normal}, and by $\dist(\cdot,M):X\rightarrow\R^+$ the $g$-distance to $M=\partial X$.

Given two quantities $A(\lambda)$, $B(\lambda)$ depending on $\lambda$, we write $A(\lambda)\sim B(\lambda)$ as $\lambda\to+\infty$ if $\lim_{\lambda\to+\infty}\frac{A(\lambda)}{B(\lambda)}=1$.

\begin{theorem}\label{thm:main}
We have two cases, according to whether $\bm> \beta_c$ or $\bm\leq\beta_c$:
\begin{enumerate}
\item If $\bm> \beta_c$ then, taking an arbitrary $\bar{\beta} \in (\beta_c, \bm)$, 
\begin{equation}\label{eq:th_beta_groesser_betac}
N(\lambda )\sim \int_{\{m\in M\,\mid\, \beta(m)\geq\bar{\beta}\}} A(\beta(m),n)\lambda^{\frac{n}{4}(2+\beta(m))}\,d\vol_{h_0}(m)
\end{equation}
as $\lambda \rightarrow +\infty$.
The asymptotics does not depend on the choice of $\bar{\beta}$. 
\item If $\bm\leq  \beta_c$ then 
$$
N(\lambda )\sim \gamma_{n+1}\,\lambda^{\frac{n+1}{2}}\vol_g\big(\{\dist(\cdot,M)\geq\lambda^{-1/2}\}\big)
$$
as $\lambda \rightarrow +\infty$.
\end{enumerate}
\end{theorem}

When $\beta$ is constant, we recover the result of Theorem \ref{theo:main0}, as expected.
In \eqref{eq:th_beta_groesser_betac}, note that $\frac{n}{4}(2+\beta(m))=d_H(\beta(m))/2$, where $d_H(\beta)=n(1+\beta/2)$. 

It is likely that Theorem \ref{thm:main} remains valid under lower regularity assumptions on $\bar g$ and $\alpha$, for instance if $\bar g$ is continuous and $\alpha$ is H\"older continuous. Indeed, the differentiability of $\alpha$ is only used later in the proof of Theorem \ref{theo:normal}. We leave this extension as an open problem.

\medskip
When $\beta$ is of class $C^2$ near its maxima and non-degenerate in the sense of Morse-Bott, we have the following corollary.

\begin{corollary}\label{co:codimension}
Assume that $\beta\in C^2(M)$ achieves its maximum $\bm$ on a compact submanifold
$W = \{m\in M\ \mid\ \beta(m)=\bm\}$ of codimension $d$. Assume moreover that $\beta$ is Morse-Bott along $W$, that is, the restriction of $-\mathrm{Hess}(\beta)$ to the normal bundle $NW$ is positive definite. For $w\in W$, denote by
$$
Q(w)=-\mathrm{Hess}(\beta)(w)\vert_{N_wW}
$$
this positive definite quadratic form, and define $\det Q(w)$ using the $h_0$-orthonormal identification of $N_wW$ with $\R^d$.
\begin{enumerate}
\item If $\bm>\beta_c=\frac{2}{n}$, then
$$
N(\lambda)\sim C_W\lambda^{\frac{n}{2}+\frac{n}{4}\bm}(\ln\lambda)^{-\frac{d}{2}},
\qquad
C_W=A(\bm,n)(2\pi)^{\frac{d}{2}}\Big(\frac{n}{4}\Big)^{-\frac{d}{2}}
\int_W \frac{d\vol_{h_0\vert_W}(w)}{\sqrt{\det Q(w)}}.
$$
\item If $\bm=\beta_c=\frac{2}{n}$, then
\begin{itemize}
\item if $d=1$,
$$
N(\lambda)\sim C_W\lambda^{\frac{n+1}{2}}\sqrt{\ln\lambda},
\qquad
C_W=\gamma_{n+1}\,2\sqrt{\frac{2\pi}{n}}\int_W \frac{d\vol_{h_0\vert_W}(w)}{\sqrt{\det Q(w)}},
$$
\item if $d=2$,
$$
N(\lambda)\sim C_W\lambda^{\frac{n+1}{2}}\ln(\ln\lambda),
\qquad
C_W=\gamma_{n+1}\,\frac{4\pi}{n}\int_W \frac{d\vol_{h_0\vert_W}(w)}{\sqrt{\det Q(w)}},
$$
\item if $d>2$,
$$
N(\lambda)\sim\gamma_{n+1}\vol_g(X)\lambda^{\frac{n+1}{2}}.
$$
\end{itemize}
\item If $\bm<\beta_c=\frac{2}{n}$, then $\vol_g(X)<+\infty$ and
$$
N(\lambda)\sim\gamma_{n+1}\,\vol_g(X)\,\lambda^{\frac{n+1}{2}}.
$$
\end{enumerate}
\end{corollary}

Corollary \ref{co:codimension} is proved in Section \ref{proof_cor_co:codimension} by using the Laplace method (see Appendix \ref{app:gauss}) to compute the integrals involved in Theorem \ref{thm:main}.

\paragraph{Novelty and comparison with the literature.}
For a constant exponent $\alpha\equiv\alpha_0$ (equivalently $\beta\equiv\beta_0$), the Weyl asymptotics of $\triangle_g$ is known and exhibits three regimes separated by the critical value $\beta_c=\frac{2}{n}$; see \cite{CDHDT24}. A general mechanism expressing Weyl laws in terms of truncated volumes for singular manifolds was developed in \cite{CPR24}.

The present paper treats the genuinely variable exponent situation: the singularity order changes along $M$ and the local Hausdorff dimension $d_H(m)=n\big(1+\frac{\beta(m)}{2}\big)$ is no longer constant. We show that, when $\bm>\beta_c$, the eigenvalue counting function is governed by the region where $\beta$ is close to $\bm$ and satisfies an integral Weyl law (Theorem \ref{thm:main}, item 1). In the subcritical and critical range $\bm\leq\beta_c$, the Weyl law is governed by the truncated volume $\vol_g(\{\dist(\cdot,M)>\lambda^{-1/2}\})$ (Theorem \ref{thm:main}, item 2), leading in particular to logarithmic corrections at $\bm=\beta_c$ when the maximum set is Morse-Bott (Corollary \ref{co:codimension}). The boundary-dominated supercritical regime may be viewed as an effective-dimension phenomenon, conceptually close to non-classical Weyl laws for degenerate or hypoelliptic operators where the spectral growth is dictated by an anisotropic dimension (see, e.g., \cite{MS78}). 
The constant-exponent case has a parallel weighted-domain history going back to \cite{VS74} in which Weyl asymptotics was established for second-order elliptic operators degenerating on the boundary of a smooth domain $\Omega\subset\R^n$ with a fixed degeneracy order. Our framework is intrinsic and geometric: the singularity is encoded as a Riemannian metric degeneration along a boundary submanifold $M$, treated through the Friedrichs Laplacian of the singular metric. The variable-exponent regime $\alpha=\alpha(m)$ studied here lies outside the framework of~\cite{VS74}, and to our knowledge has not been treated previously in this Weyl-law setting.

\paragraph{Examples and motivation.}
Variable boundary degeneracy exponents arise naturally in several geometric and physical settings.

\smallskip
\noindent
\emph{Variable-exponent Grushin operators.} The standard Grushin operator $\mathcal G_k=-\partial_x^2-x^{2k}\triangle_y$ on $(0,1)\times\mathbb{T}^n$, with $k>0$ fixed, is unitarily equivalent (in the model regime) to the Friedrichs Laplacian of the singular metric $dx^2+x^{-2k}\vert dy\vert^2$, that is, $g_\beta$ with $\beta=2k$, equivalently $\alpha=2k/(k+1)$. The natural variable-exponent generalization $\mathcal G_{k(\cdot)}=-\partial_x^2-x^{2k(y)}\triangle_y$, where $k:\mathbb{T}^n\to(0,+\infty)$ is smooth, corresponds precisely to a metric of the form $u^{-\alpha(y)}\bar g$ with $\alpha(y)=2k(y)/(k(y)+1)$ varying along the singular stratum $M=\{u=0\}$. Such operators are natural in models of variable subellipticity where the order of vanishing of the transverse vector fields depends on the position along the characteristic manifold.

\smallskip
\noindent
\emph{Acoustic modes in gas giants with non-uniform surface conditions.} In the model studied in~\cite{CDHDT24}, the constant exponent $\alpha$ encodes a fixed polytropic equation of state at the surface of the planet. When the thermodynamic profile or chemical composition of the boundary layer varies with latitude or longitude (as it does in any realistic atmospheric model) the equation of state varies along the surface, producing a boundary-dependent degeneracy order $\alpha=\alpha(m)$. The constant-exponent assumption of~\cite{CDHDT24} is thus an idealization; the variable-exponent setting addressed here is the natural framework for such physical models, and a prerequisite for extracting spectral signatures of non-uniform boundary structure.

\smallskip
\noindent
\emph{Sub-Riemannian and singular geometries with variable degeneracy.} On a compact manifold equipped with a regular sub-Riemannian structure away from a singular stratum $M$ at which the distribution becomes non-equiregular, the order of degeneracy of the intrinsic sub-Laplacian transverse to $M$ may depend on the boundary point when the step of the distribution varies along $M$. While the precise identification with our framework requires a careful unitary equivalence at the level of weighted spaces, the variable-exponent regime is the natural spectral counterpart of variable-step degeneracy along a singular stratum.

\subsection{Phase transition at the critical exponent $\beta=\beta_c$}\label{sec:phase+}
This subsection is purely interpretative and is not used in the sequel. Its purpose is to clarify the borderline case $\beta=\beta_c$ in the constant-exponent model and to contextualize the logarithmic correction in Theorem \ref{theo:main0} and in Corollary \ref{co:codimension}.

\paragraph{Constant exponent model, pole of $A(\beta,n)$, and emergence of $\ln\lambda$.}
Consider the constant exponent metric $g_\beta=dx^2+x^{-\beta}h_0$ on $(0,1)\times M$, with $\dim M=n$. Theorem \ref{theo:main0} shows that the order of growth of $N_{g_\beta}(\lambda)$ changes at $\beta_c=\frac{2}{n}$: it is classical for $\beta<\beta_c$, boundary-dominated and non-classical for $\beta>\beta_c$, and exhibits a logarithmic factor at $\beta=\beta_c$. In particular, the supercritical asymptotic cannot be uniform as $\beta\downarrow\beta_c$.

A convenient way to express the compatibility between the supercritical and critical formulas is to use the fact, proved in \cite{CDHDT24}, that the supercritical coefficient $A(\beta,n)$ has a simple pole at $\beta_c$, namely,
\begin{equation}\label{pole-A}
A(\beta,n)=\frac{2}{n}\frac{\gamma_{n+1}}{\beta-\beta_c}+\mathrm{O}(1),
\end{equation}
as $\beta\downarrow\beta_c$. This follows from the zeta-function representation $A(\beta,n)=\gamma_n\,\zeta_{\beta,n}(d_H/2)$ and the meromorphic continuation of $\zeta_{\beta,n}$.

Write $d_H=n(1+\frac{\beta}{2})$, so that $\frac{d_H}{2}=\frac{n+1}{2}+\frac{n}{4}(\beta-\beta_c)$.
For every fixed $\beta>\beta_c$, one may replace $\lambda^{d_H/2}$ by $\lambda^{d_H/2}-\lambda^{(n+1)/2}$ in the supercritical Weyl law, since $\lambda^{(n+1)/2}=\mathrm{o}(\lambda^{d_H/2})$ as $\lambda\to+\infty$. For every fixed $\lambda>1$, we have
$$
\frac{\lambda^{d_H/2}-\lambda^{(n+1)/2}}{\beta-\beta_c}
=\lambda^{(n+1)/2}\frac{\lambda^{\frac{n}{4}(\beta-\beta_c)}-1}{\beta-\beta_c}
\longrightarrow \frac{n}{4}\lambda^{(n+1)/2}\ln\lambda
$$
as $\beta\downarrow\beta_c$. Together with \eqref{pole-A}, this yields, at the level of leading terms,
$$
A(\beta,n)\big(\lambda^{d_H/2}-\lambda^{(n+1)/2}\big)
\longrightarrow \frac{\gamma_{n+1}}{2}\lambda^{(n+1)/2}\ln\lambda,
$$
which matches the critical Weyl law in Theorem \ref{theo:main0} (after multiplying by $\vol_{h_0}(M)$).

\paragraph{Variable exponent and spectral transition.}
If $\bm>\beta_c$, the leading term in Theorem \ref{thm:main}, item 1, is produced by the region where $\beta$ is close to $\bm$. The pole \eqref{pole-A} shows that the contribution of points where $\beta$ is close to $\beta_c$ is delicate and justifies the regime splitting at $\beta_c$ in Theorem \ref{thm:main}.
When $\bm\leq\beta_c$, the supercritical region is empty and the relevant quantity becomes the truncated volume $\vol_g(\{\dist(\cdot,M)>\lambda^{-1/2}\})$ (Theorem \ref{thm:main}, item 2). In the critical case $\bm=\beta_c$, this truncated volume diverges slowly, leading to the logarithmic term described in Theorem \ref{theo:main0}, item 2.

Approaching the threshold $\beta_c$ from below yields a pole coming from the volume. Indeed, for the constant exponent model,
$$
\vol_{g_\beta}((0,1)\times M)=\int_0^1\int_M x^{-\frac{n\beta}{2}}\,d\vol_{h_0}(m)\,dx
=\frac{1}{1-\frac{n\beta}{2}}\vol_{h_0}(M),
$$
which diverges like $\frac{2}{n}(\beta_c-\beta)^{-1}\vol_{h_0}(M)$ as $\beta\uparrow\beta_c$. In the variable exponent setting, the critical behavior is governed by the geometry of $\{\beta=\beta_c\}$, which may produce $\sqrt{\ln\lambda}$ or $\ln\ln\lambda$ corrections (Corollary \ref{co:codimension}, item 2).

\begin{remark}[A separation-of-variables heuristic on the infinite cone]
Let $(\mu_j)_{j\geq 1}$ be the eigenvalues of $\triangle_M$ on $(M,h_0)$. On the infinite cone $M\times(0,+\infty)$ with metric $g_\beta$, separation of variables reduces the analysis to the one-dimensional Schr\"odinger operators $P_{\mu_j}=-\partial_x^2+C(n,\beta)x^{-2}+\mu_j x^\beta$.
If $(E_{j,k})_{k\geq 1}$ denotes the spectrum of $P_{\mu_j}$, the scaling $x=\mu_j^{-\frac{1}{\beta+2}}s$ gives $E_{j,k}=\mu_j^{\frac{2}{\beta+2}}\nu_k$, where $(\nu_k)_{k\geq 1}$ is the spectrum of $P_1=-\partial_x^2+C(n,\beta)x^{-2}+x^\beta$. 
Therefore, at a heuristic level,
$$
N(\lambda)\sim \sum_{k\geq 1}\#\bigg\{j\geq 1 \ \ \big\vert\ \ \mu_j\leq \left(\frac{\lambda}{\nu_k}\right)^{\frac{\beta+2}{2}}\bigg\} ,
$$
as $\lambda\to+\infty$. Using the Weyl law on $(M,h_0)$,
$$
\#\{j\geq 1\,\mid\, \mu_j\leq \Lambda\}\sim \gamma_n\vol_{h_0}(M)\Lambda^{\frac{n}{2}},
$$
as $\Lambda\to+\infty$, we obtain
$$
N(\lambda)\sim 
\gamma_n\vol_{h_0}(M)\,\lambda^{\frac{d_H}{2}}\sum_{k\geq 1}\nu_k^{-\frac{d_H}{2}},
$$
as $\lambda\to+\infty$, with $d_H=n(1+\frac{\beta}{2})$.
Finally, one-dimensional semiclassical analysis, equivalently the WKB or Bohr-Sommerfeld formula, gives $\nu_k\sim c\,k^{\frac{2\beta}{\beta+2}}$ as $k\to+\infty$, hence $\nu_k^{-\frac{d_H}{2}}\sim c'\,k^{-\frac{n\beta}{2}}$.
The exponent $\frac{n\beta}{2}$ crosses $1$ exactly at $\beta=\beta_c=\frac{2}{n}$.
This explains heuristically why the series converges for $\beta>\beta_c$, diverges logarithmically at $\beta=\beta_c$, and diverges polynomially for $\beta<\beta_c$.
\end{remark}


\section{Proof of Theorem \ref{thm:main}}

\paragraph{Strategy of the proof.}
We first reduce the geometry near the boundary to a model form. By Theorem \ref{theo:normal} (Section \ref{sec:normal}), in a collar neighborhood $X_a\simeq(0,a)\times M$ the metric $g$ is quasi-isometric to $g_0=dx^2+x^{-\beta(m)}h_0$, where $x\sim\dist_g(\cdot,M)$ and $h_0$ is a $C^1$ metric on $M$. By quasi-isometry stability (Appendix \ref{app:quasi}), it suffices to establish the Weyl asymptotics for $g_0$.

The argument then splits at the critical value $\beta_c=\frac{2}{n}$.

If $\bm>\beta_c$, the spectrum is dominated by the boundary neighborhood. We keep a fixed collar $X_b$ and localize tangentially, that is, on $M$, by a mesh of simplices of size $\eta=\lambda^{\gamma-\frac{1}{2}}$ with $\gamma\in(0,\frac{1}{2})$. On each simplex, $\beta$ and $h_0$ may be frozen (Proposition \ref{prop:normal-quasi} in Section \ref{sec:normal} and Lemma \ref{lem:freezing-simplex} in Section \ref{sec_bm>betac}), so the local counting function reduces to the constant exponent model. Summing over the simplices yields the integral Weyl law of Theorem \ref{thm:main}, item 1.

If $\bm\leq\beta_c$, the leading term has classical order and is governed by the truncated volume $\vol_g(\{\dist(\cdot,M)>\lambda^{-1/2}\})$ (Theorem \ref{thm:main}, item 2). In the critical case $\bm=\beta_c$, this produces logarithmic corrections. Following the truncated volume approach of \cite{CPR24}, we also localize in the transverse variable at the wavelength scale: after choosing $C_\varepsilon>0$ large enough, we set $\zeta=C_\varepsilon\lambda^{-1/2}$ and slice the collar into $x$-intervals of length $\zeta$ on which the coefficient $x^{-\beta(m)}$ is essentially constant. Standard Weyl estimates on the resulting boxes may then be summed.

\paragraph{Small parameters used in the proof.}
The small parameters appear in the proof in a fixed order. We first fix $\varepsilon\in(0,1)$, which determines all quasi-isometry constants, bracketing errors, and the mesoscopic constant $C_\varepsilon$ in the definition $\zeta=C_\varepsilon\lambda^{-1/2}$. Then, for each large $\lambda$, we set $\eta=\lambda^{\gamma-\frac{1}{2}}$ with a fixed $\gamma\in(0,\frac{1}{2})$, and we use the hierarchy
$$
\zeta=C_\varepsilon\lambda^{-1/2}\ll\eta=\lambda^{\gamma-\frac{1}{2}}\ll1,
$$
together with $\eta\vert\ln\zeta\vert\to0$ as $\lambda\to+\infty$, to justify the successive freezing and quasi-isometric reductions. The latter is a consequence of the choice of scales: since $\eta\vert\ln\zeta\vert=\lambda^{\gamma-\frac12}\,\vert\ln(C_\varepsilon\lambda^{-1/2})\vert\sim\frac12\lambda^{\gamma-\frac12}\ln\lambda$ and $\gamma<\frac12$, this tends to $0$.

\subsection{Quasi-isometric normal form} \label{sec:normal}

\begin{theorem}\label{theo:normal} 
There exists a unique $C^1$ Riemannian metric $h_0$ on $M$ such that, if $g_0 = dx^2 +x^{-\beta} h_0 $ in a collar neighborhood $X_a = (0,a]_x\times M$ of $M$, then, for every $\varepsilon >0 $, there exists $b\leq a $ so that $g$ expressed in the coordinates $(x,m)$ and $g_0$ are $\varepsilon$-quasi-isometric (see Appendix \ref{app:quasi} for the definition) in $X_b = (0,b]_x\times M$.
Moreover $(1-\varepsilon )x \leq \dist(\cdot,M) \leq (1+\varepsilon) x$. 
\end{theorem}

\begin{proof}
Fix $m\in M$. Since $u$ is transverse to $M$, the covector $du$ is nonzero on $M$ and $\ker(du)=T_mM$. At $(0,m)$, we may write $\bar g(0,m)=A(m)^2du^2+k(m)$, where $A(m)>0$ and $k(m)$ is a positive definite quadratic form on $T_mM$. Define
$$
g_1=u^{-\alpha(m)}\bar g(0,m)=u^{-\alpha(m)}\big(A(m)^2du^2+k(m)\big).
$$
Since $\bar g$ is $C^1$ up to $M$, we have $g=u^{-\alpha}\bar g\sim g_1$ as $u\to 0$ in the sense of quasi-isometries (see Appendix \ref{app:quasi}). Set
$$
x=A(m)\int_0^u t^{-\frac{\alpha(m)}{2}}\,dt=\frac{2A(m)}{2-\alpha(m)}u^{1-\frac{\alpha(m)}{2}}.
$$
Then
$$
dx=A(m)u^{-\frac{\alpha(m)}{2}}du+\mathrm{O}\big(u^{1-\frac{\alpha(m)}{2}}\vert\ln u\vert\big)\,dm,
$$
hence the $dx\,dm$ cross term is negligible compared with the main diagonal terms as $u\to 0$. Writing
$\beta(m)=\frac{2\alpha(m)}{2-\alpha(m)}$, we compute
$$
u^{-\alpha(m)}k(m) = x^{-\beta(m)}\left(\frac{2A(m)}{2-\alpha(m)}\right)^{\beta(m)}k(m).
$$
Therefore, up to quasi-isometry,
$$
g\sim dx^2+x^{-\beta(m)}h_0(m),
\qquad
h_0(m)=\left(\frac{2A(m)}{2-\alpha(m)}\right)^{\beta(m)}k(m).
$$
Finally, since $g\sim dx^2+x^{-\beta(m)}h_0(m)$ and $x$ is a transverse coordinate, we have $\dist_g((x,m),M)\sim x$ as $x\to 0$. The uniqueness of $h_0$ follows from
$$
h_0(m)=\lim_{x\to 0} x^{\beta(m)}g\vert_{T_mM}=\lim_{x\to 0} \dist_g((x,m),M)^{\beta(m)}g\vert_{T_mM}.
$$
The theorem is proved.
\end{proof}

\begin{remark}[Boundary distance and intrinsic meaning of $h_0$]\label{rem:boundary-distance}
Assume first that $\beta$ is constant and that, in a collar, the metric is exactly $g=dx^2+x^{-\beta}h_0$.
Let $d_\partial$ be the distance induced on $M$ by the ambient distance $d_g$. Then, as $q\to p$ in $M$,
$$
d_\partial(p,q)=c_\beta \, d_{h_0}(p,q)^{\frac{2}{2+\beta}} (1+\mathrm{o}(1)),
$$
with an explicit constant $c_\beta>0$ given in Lemma \ref{lem:boundary-distance} in Appendix \ref{app:boundary-distance}.
In particular, for fixed $\beta$, the boundary metric determines $h_0$ up to a multiplicative constant depending only on $\beta$.
This gives an intrinsic pointwise meaning to $h_0$ along the boundary.

For a variable $\beta(m)$, the same blow-up argument as in Lemma \ref{lem:boundary-distance} applies after freezing the coefficients at a boundary point $p$; see Remark \ref{rem:boundary-distance-variable}. More precisely, for every $p\in M$, as $q\to p$,
$$
d_\partial(p,q)=c_{\beta(p)}\,d_{h_0}(p,q)^{\frac{2}{2+\beta(p)}}(1+\mathrm{o}(1)),
$$
where $c_{\beta(p)}$ is given by Lemma \ref{lem:boundary-distance} with $\beta$ replaced by $\beta(p)$.
In particular, there exist a neighborhood $U$ of $p$ and constants $c'_p,C'_p>0$ such that, for every $q\in U$,
$$
c'_p\,d_{h_0}(p,q)^{\frac{2}{2+\beta(p)}}\leq d_\partial(p,q)\leq C'_p\,d_{h_0}(p,q)^{\frac{2}{2+\beta(p)}}.
$$
\end{remark}

\begin{proposition}\label{prop:normal-quasi}
In a domain $[\eta_1, a]\times T_\eta \subset X_a$ where $T_\eta\subset M$ has diameter $\eta $, if we assume that $\eta\vert\ln \eta_1\vert \leq \varepsilon $, then $g_0$ is $\mathrm{O}(\varepsilon)$-quasi-isometric to $dx^2+x^{-\beta(m_0)}h_0 $ for any $m_0\in T_\eta$.
This holds in particular if $\eta =\lambda ^{\gamma -1/2} $ with $0<\gamma < 1/2$ and $\eta_1=C\lambda^{-1/2}$ for $\lambda $ large enough.
\end{proposition}

\begin{proof}
This follows from the fact that $\vert x^{\beta(m)-\beta(m_0)}-1\vert=\mathrm{O}(\varepsilon)$. 
\end{proof}

\subsection{Triangulations}\label{sec:triang}

Any  closed manifold $M$ of dimension $n$  admits  smooth triangulations; this was proved in \cite{Cairns}
and in \cite{Whit}; for a more recent proof, see the textbook \cite{Mu61}.
This means that the manifold $M$ is topologically identified with a finite simplicial complex $K$ of dimension $n$ so that each simplex of dimension $n$ of $K$ is the image under a smooth diffeomorphism $F$ of the standard simplex $T_0 = \{ x=(x_i)\in \R^n \ \mid\ 0\leq x_1 \leq x_2 \leq \cdots \leq x_n \leq 1 \}$.
More precisely, $F$ is a diffeomorphism mapping a neighborhood $U$ of $T_0$ in $\R^n$ onto an open subset $V$ of $M$.

\begin{definition}\label{def:simplex}
A \emph{flat simplex} is any image of the standard  simplex $T_0\subset \R^n$ by an affine transform of $\R^n$, equipped with the standard Euclidean metric $\vert dy\vert^2$ induced by that of $\R^n$.
\end{definition}

\begin{definition} 
We say that two flat simplices  $T_1$ and $T_2$ have the same \emph{shape} if $T_2$ is the image of $T_1$ by an isometry composed with a dilation (i.e., a similitude).
\end{definition}

\begin{definition}\label{def:polyhedral}
We say that $h$ is a \emph{polyhedral metric} on $M$ if a smooth triangulation $\mathcal T$ of $M$ is given and if, for every simplex $T\in\mathcal T$, $(T,h\vert_T)$ is isometric to a flat simplex in the sense of Definition \ref{def:simplex}. We do not require that the restrictions match along faces.
\end{definition}

In other words, a polyhedral metric on $M$ is a possibly discontinuous locally flat metric on $M$.

\begin{proposition}\label{prop:quasi}
Given any $\varepsilon >0 $ and any continuous metric $h$ on $M$, there exists a polyhedral metric $h_\varepsilon $ that is $\varepsilon $-quasi-isometric to $(M,h)$ (see Definition \ref{defi:quasi}). 
\end{proposition}

In what follows, for a fixed value of $\varepsilon>0$, the polyhedral metric $h_\varepsilon$ will be used in order to get approximate Weyl laws.

\begin{proof}[Proof of Proposition \ref{prop:quasi}.]
Fix $\varepsilon>0$. Since $h$ is continuous and $M$ is compact, there exists $\delta>0$ such that for any $m,m'\in M$ with $\dist_h(m,m')\leq\delta$, one has
$$
\frac{1}{1+\varepsilon}h(m)\leq h(m')\leq(1+\varepsilon)h(m)
$$
as quadratic forms on $TM$. Choose a smooth triangulation $\mathcal T$ of $M$ with mesh size at most $\delta$. For each simplex $T\in\mathcal T$, pick a point $m_T\in T$ and define on $T$ the constant coefficient metric obtained by freezing $h$ at $m_T$ in barycentric coordinates. This yields a polyhedral metric $h_\varepsilon$ which is $\varepsilon$-quasi-isometric to $h$.
\end{proof}

According to a result of \cite{Fr} (see also \cite{EG00}), given any $k\in \N^*$, a flat simplex $T$ in $\R^n$ can be subdivided into $k^n$ simplices of the same volume so that the number of possible shapes is less than $n!$.

Hence, given any $\eta>0$  small enough (in what follows, $\eta$ will tend to $0$ while $\lambda\to+\infty$), we can again subdivide the simplices of the polyhedral metric $h_\varepsilon$ into simplices $T_\eta$ whose  diameters lie in the interval $[c\eta,\eta] $ for some fixed $c\in(0,1)$ not depending on $\eta$.
Let us summarize.

\begin{proposition}\label{prop:triang}
Let $h$ be a continuous metric on $M$. For any $\varepsilon>0$ and $\eta>0$, there exists a polyhedral metric $h_\eta$ on $M$ that is $\varepsilon$-quasi-isometric to $h$ and a triangulation $\mathcal T_\eta$ of $M$ such that the diameter of each simplex $T_\eta\in\mathcal T_\eta$ lies in $[c\eta,\eta]$ for some $c\in(0,1)$ not depending on $\eta$, and such that only finitely many simplex shapes occur, with a bound independent of $\eta$.
\end{proposition}

\begin{proof}
Apply Proposition \ref{prop:quasi} to obtain a polyhedral metric $h_\varepsilon$ on a fixed triangulation $\mathcal T$ of $M$. Subdivide the simplices of $\mathcal T$ by iterating the Freudenthal subdivision (see \cite{Fr}) until all simplices have diameter at most $\eta$. The diameter bounds are part of the construction, and the number of shapes is bounded by $n!$ times the number of simplices of the fixed triangulation $\mathcal T$, hence independently of $\eta$.
\end{proof}

\begin{remark}\label{rem:finite-shapes}
Proposition \ref{prop:triang} provides a finite family of reference shapes. In particular, any constant in a local spectral estimate that depends only on the shape (for instance Lemma \ref{le:uniform}) can be chosen uniformly for all simplices $T_\eta$ and for all large $\lambda$. This uniformity is used repeatedly when summing local contributions.
\end{remark}
 
In particular, the number of shapes is bounded by $n!$ times the number of simplices in the first subdivision
with the metric $h_\varepsilon$. This bound is uniform in $\eta$.

\subsection{Hardy's inequalities }\label{sec:hardy}

\begin{lemma}\label{le:0ev}
There exists $c>0$ not depending on $\beta $ such that the operator $\triangle_X$ restricted to the domain $X_{ c \lambda ^{-1/2}} $ with Dirichlet boundary conditions at $x=0$ and Dirichlet or Neumann boundary conditions at $c \lambda ^{-1/2}$ has no eigenvalues smaller than $\lambda $.
\end{lemma}

\begin{proof}
On the model metric $g_0=dx^2+x^{-\beta}h_0$ with constant $\beta$, the Friedrichs Laplacian is unitarily equivalent to $-\partial_x^2+C(n,\beta)x^{-2}+x^\beta\triangle_M$ (see \cite{CDHDT24}). Since $C(n,\beta)x^{-2}\geq 0$ and $x^\beta\triangle_M\geq 0$, we have the form inequality
$$
-\partial_x^2+C(n,\beta)x^{-2}+x^\beta\triangle_M\geq -\partial_x^2.
$$
Hence, on $(0,c\lambda^{-1/2})\times M$ with a Dirichlet condition at $x=0$, the lowest eigenvalue is bounded from below by the lowest eigenvalue of $-\partial_x^2$ on $(0,c\lambda^{-1/2})$ with the same boundary condition at $x=c\lambda^{-1/2}$ (Dirichlet or Neumann). This lowest eigenvalue equals $\pi^2c^{-2}\lambda$ in the Dirichlet case and $\frac{\pi^2}{4}c^{-2}\lambda$ in the Neumann case. Choosing $0<c<\pi/2$ gives a lower bound strictly larger than $\lambda$ in both
the Dirichlet and Neumann cases, for $\lambda$ large enough.
\end{proof}

We next record a uniform Weyl estimate on mesoscopic domains, which will be used repeatedly in the localization argument.

\subsection{Weyl estimates for ``mesoscopic'' domains}\label{sec:weyl-box}
One of the key ingredients is to get a uniform Weyl estimate for domains of small size $\eta$.
The relevant sizes are the \emph{mesoscopic lengths}: if $\mu$ is a spectral parameter, then the typical wavelength is $\mu^{-1/2}$, and we only expect good uniform estimates on domains whose size is much larger than $\mu^{-1/2}$.
We will therefore use rescaled domains $D_\eta=\eta D_1\subset\R^d$, where $d$ may vary according to the application.
Recall the rescaling laws $\mu_j(D_\eta)=\eta^{-2}\mu_j(D_1)$ and $\vert D_\eta\vert=\eta^d\vert D_1\vert$.

\begin{lemma}\label{le:uniform}
Let $d\geq1$, let $D_1\subset\R^d$ be a piecewise smooth bounded open set and let $D_\eta=\eta D_1$ for any $\eta>0$. Let $\varepsilon>0$. There exists $C_\varepsilon>0$, depending on $D_1$, such that for any $\eta>0$ and any $\mu\geq\frac{C_\varepsilon}{\eta^2}$, we have, for both Dirichlet and Neumann Laplacians on $D_\eta$,
\begin{equation*}
\frac{1}{1+\varepsilon}\gamma_d\vert D_\eta\vert\mu^{d/2}\leq N_{D_\eta}^{D/N}(\mu)\leq(1+\varepsilon)\gamma_d\vert D_\eta\vert\mu^{d/2}.
\end{equation*}
Moreover, for any $j\geq C_\varepsilon$, for both Dirichlet and Neumann boundary conditions, we have
$$
\frac{1}{1+\varepsilon}\left(\frac{j}{\gamma_d\vert D_\eta\vert}\right)^{2/d}\leq \mu_j^{D/N}(D_\eta)\leq (1+\varepsilon)\left(\frac{j}{\gamma_d\vert D_\eta\vert}\right)^{2/d},
$$
where $\gamma_d=\frac{b_d}{(2\pi)^d}$ and $b_d$ denotes the Euclidean volume of the unit ball in $\R^d$.
\end{lemma}

\begin{proof}
This follows from $\mu_j(D_\eta)=\eta^{-2}\mu_j(D_1)$ and $\vert D_\eta\vert=\eta^d\vert D_1\vert$, combined with the Weyl law for $D_1$.
\end{proof}

\subsection{The case $\bm > \beta_c$}\label{sec_bm>betac}

\begin{lemma}[Freezing on a simplex]\label{lem:freezing-simplex}
Assume that $\beta\in C^1(M)$ and let $T_\eta$ be a simplex of diameter $\mathrm{O}(\eta)$. Pick $m_T\in T_\eta$. If $\eta=\lambda^{\gamma-\frac{1}{2}}$ with $\gamma\in(0,\frac{1}{2})$, then as $\lambda\to+\infty$ we have uniformly for $m\in T_\eta$:
$$
\lambda^{\frac{n}{4}\beta(m)}=\lambda^{\frac{n}{4}\beta(m_T)}\big(1+\mathrm{o}(1)\big),
\qquad
A(\beta(m),n)=A(\beta(m_T),n)+\mathrm{o}(1),
$$
and, for $x\in[c\lambda^{-1/2},b]$,
$$
x^{-\beta(m)}=x^{-\beta(m_T)}\big(1+\mathrm{o}(1)\big).
$$
\end{lemma}

\begin{proof}[Proof of Lemma \ref{lem:freezing-simplex}.]
Since $\beta$ is Lipschitz on $M$, $\vert\beta(m)-\beta(m_T)\vert\leq C\eta$. Hence
$$
\frac{\lambda^{\frac{n}{4}\beta(m)}}{\lambda^{\frac{n}{4}\beta(m_T)}}=\exp\Big(\frac{n}{4}(\beta(m)-\beta(m_T))\ln\lambda\Big)=\exp\big(\mathrm{O}(\eta\ln\lambda)\big)=1+\mathrm{o}(1),
$$
because $\eta\ln\lambda=\lambda^{\gamma-\frac{1}{2}}\ln\lambda\to 0$. The statements for $A(\beta,n)$ and for $x^{-\beta(m)}$ are proved similarly, using $x\geq c\lambda^{-1/2}$ and $\vert\ln x\vert\leq C\ln\lambda$.
\end{proof}

The next lemma, proved in Appendix \ref{app:truncation}, addresses the only place where the proof below cannot rely directly on Theorem \ref{theo:main0}: after freezing the coefficients on a simplex, one is led to a constant-exponent product cell truncated at the wavelength scale $x=c\lambda^{-1/2}$. After separation of variables and scaling, the lower endpoint of the radial interval is proportional to $c$, so $c$ must be taken arbitrarily small in order to reduce to the model operator $P_1(\beta)$ on $(0,+\infty)$.

\begin{lemma}\label{lem_constant-cell}
Fix $\bar\beta\in(\beta_c,+\infty)$, $b>0$, and a finite family $\mathscr S$ of flat simplex shapes in $\R^n$. Let $T$ be a flat simplex whose shape belongs to $\mathscr S$, and let $\beta\in[\bar\beta,\bm]$. Consider the product metric $g_{\beta,T}=dx^2+x^{-\beta}\vert dy\vert^2$ on $\Omega_T=(0,b)\times T$. Let $N_{\Omega_T}^{D/N}(\lambda)$ denote the counting function of the corresponding Friedrichs Laplacian, with Dirichlet or Neumann boundary conditions imposed on $\{x=b\}\cup((0,b)\times\partial T)$. Then
$$
N_{\Omega_T}^{D/N}(\lambda) = A(\beta,n)\, \vert T\vert\,\lambda^{\frac{n}{2}+\frac{n}{4}\beta}\, (1+\mathrm{o}(1))
$$
as $\lambda\to+\infty$, uniformly with respect to $\beta\in[\bar\beta,\bm]$ and to flat simplexes $T$ whose shape belongs to $\mathscr S$ and whose diameter is comparable to $\eta=\lambda^{\gamma-\frac{1}{2}}$, with $\gamma\in(0,\frac{1}{2})$, uniformly with respect to the comparison constants coming from Proposition \ref{prop:triang}.
Moreover, for every $\varepsilon>0$, there exists $c_\varepsilon\in(0,\pi/2)$ such that, for every $c\in(0,c_\varepsilon)$, setting
$$
\Omega_{T,\lambda,c}=[c\lambda^{-1/2},b]\times T,
$$
and denoting by $N_{\Omega_{T,\lambda,c}}^{D/N}(\lambda)$ the counting function with the same boundary conditions and with Dirichlet or Neumann boundary conditions at $x=c\lambda^{-1/2}$, one has
$$
N_{\Omega_{T,\lambda,c}}^{D/N}(\lambda)
=
A(\beta,n)\,\vert T\vert\,\lambda^{\frac{n}{2}+\frac{n}{4}\beta}
\big(1+\mathrm{O}(\varepsilon)+\mathrm{o}(1)\big)
$$
as $\lambda\to+\infty$, uniformly with respect to $\beta\in[\bar\beta,\bm]$ and to flat simplexes $T$ whose shape belongs to $\mathscr S$ and whose diameter is comparable to $\eta=\lambda^{\gamma-\frac{1}{2}}$, with $\gamma\in(0,\frac{1}{2})$, uniformly with respect to the comparison constants coming from Proposition \ref{prop:triang}.
\end{lemma}

\begin{proof}[Proof of the first item of Theorem \ref{thm:main}.]
Fix $\bar\beta\in(\beta_c,\bm)$ and $\varepsilon\in(0,1)$. 
Choose $b\in(0,\min(a,1))$ as in Section \ref{sec:normal}.
By Theorem \ref{theo:normal} and Proposition \ref{prop:normal-quasi}, we may replace $g$ by the model metric $g_0=dx^2+x^{-\beta(m)}h_0$ on $X_b=(0,b)\times M$. We also apply Proposition \ref{prop:triang} to $h_0$ with mesh size $\eta=\lambda^{\gamma-\frac{1}{2}}$, $\gamma\in(0,\frac{1}{2})$: this gives a triangulation $\mathcal T_\eta$ of $M$ into simplices of diameter at most $\eta$ and a polyhedral metric $h_\eta$ which is $\varepsilon$-quasi-isometric to $h_0$. Let $\mathscr S$ be the corresponding finite family of simplex shapes.

Let $c_\varepsilon$ be given by Lemma \ref{lem_constant-cell} for this family $\mathscr S$, and choose $c\in(0,c_\varepsilon)$. In particular $c<\pi/2$, so Lemma \ref{le:0ev} applies.

For $\lambda$ large, Lemma \ref{le:0ev} shows that the inner collar $X_{c\lambda^{-1/2}}$ contributes no eigenvalue below $\lambda$. Hence, by Dirichlet-Neumann bracketing across the artificial boundary $\{x=c\lambda^{-1/2}\}$,
$$
N_{X\setminus X_{c\lambda^{-1/2}}}^D(\lambda)\leq N_X(\lambda)\leq N_{X\setminus X_{c\lambda^{-1/2}}}^N(\lambda).
$$
Splitting again at the fixed hypersurface $\{x=b\}$ and applying bracketing once more, we obtain
$$
N_{X_b\setminus X_{c\lambda^{-1/2}}}^D(\lambda)+N_{X\setminus X_b}^D(\lambda)\leq N_X(\lambda)\leq N_{X_b\setminus X_{c\lambda^{-1/2}}}^N(\lambda)+N_{X\setminus X_b}^N(\lambda).
$$
Since $X\setminus X_b$ is a smooth compact manifold with boundary, the classical Weyl law gives
$$
N_{X\setminus X_b}^N(\lambda)=\bigO(\lambda^{\frac{n+1}{2}})=\mathrm{o}(\lambda^{\frac{n}{2}+\frac{n}{4}\bar\beta}),
$$
as $\lambda\to+\infty$, because $\bar\beta>\beta_c=\frac{2}{n}$. Therefore it is enough to estimate the truncated collar $X_b\setminus X_{c\lambda^{-1/2}}$ with Dirichlet or Neumann conditions on the artificial boundaries.

For each simplex $T_\eta\in\mathcal T_\eta$, set $\Omega_{T_\eta,\lambda,c}=[c\lambda^{-1/2},b]\times T_\eta$.
Denote by $N_{X_b\setminus X_{c\lambda^{-1/2}}}^{D/N}(\lambda)$ the counting function corresponding either to Dirichlet or to Neumann boundary conditions. Applying Dirichlet-Neumann bracketing across the internal faces of the triangulation, we obtain
$$
\sum_{T_\eta\in\mathcal T_\eta}N_{\Omega_{T_\eta,\lambda,c}}^D(\lambda)\leq N_{X_b\setminus X_{c\lambda^{-1/2}}}^{D/N}(\lambda)\leq\sum_{T_\eta\in\mathcal T_\eta}N_{\Omega_{T_\eta,\lambda,c}}^N(\lambda).
$$

\noindent
\emph{Simplices with $\beta(m_{T_\eta})\geq\bar\beta$.}
Let $T_\eta$ be a simplex with basepoint $m_{T_\eta}$ such that $\beta(m_{T_\eta})\geq\bar\beta$. Since $\eta\vert\ln(c\lambda^{-1/2})\vert\to 0$, Proposition \ref{prop:normal-quasi} allows us to freeze $\beta$ on $T_\eta$ up to a factor $1+\bigO(\varepsilon)$, and Lemma \ref{lem_constant-cell} yields, uniformly in $T_\eta$,
\begin{equation}\label{eq:groesser_als_betac}
N_{\Omega_{T_\eta,\lambda,c}}^{D/N}(\lambda)=A(\beta(m_{T_\eta}),n)\,\vol_{h_\eta}(T_\eta)\,\lambda^{\frac{n}{2}+\frac{n}{4}\beta(m_{T_\eta})} \big(1+\bigO(\varepsilon)+\mathrm{o}(1)\big).
\end{equation}

\noindent
\emph{Simplices with $\beta(m_{T_\eta})<\bar\beta$.}
If $\beta(m_{T_\eta})<\bar\beta$, then on $[c\lambda^{-1/2},b]$ one has $x^{\beta(m_{T_\eta})}\geq x^{\bar\beta}$ since $b<1$, and of course $C(n,\beta(m_{T_\eta}))x^{-2}\geq 0$. Therefore,
$$
-\partial_x^2+C(n,\beta(m_{T_\eta}))x^{-2}+x^{\beta(m_{T_\eta})}\triangle_{T_\eta}\geq -\partial_x^2+x^{\bar\beta}\triangle_{T_\eta}.
$$
The operator on the right is even simpler than the constant exponent model with exponent $\bar\beta$, and a similar separation of variables argument as in Theorem \ref{theo:main0} gives the uniform bound
\begin{equation}\label{eq:kleiner_als_betac}
N_{\Omega_{T_\eta,\lambda,c}}^{D/N}(\lambda)=\bigO\Big(\vol_{h_\eta}(T_\eta)\lambda^{\frac{n}{2}+\frac{n}{4}\bar\beta}\Big).
\end{equation}
In particular, no extra assumption $\alpha>0$ is needed: if $\beta(m_{T_\eta})=0$, the inverse-square term simply disappears.

Summing \eqref{eq:groesser_als_betac} over the simplices with $\beta(m_{T_\eta})\geq\bar\beta$ and using \eqref{eq:kleiner_als_betac} for the remaining ones, we get
$$
N_{X_b\setminus X_{c\lambda^{-1/2}}}^{D/N}(\lambda) = \sum_{\stackrel{T_\eta\in\mathcal T_\eta}{\beta(m_{T_\eta})\geq\bar\beta}} A(\beta(m_{T_\eta}),n)\,\vol_{h_\eta}(T_\eta)\,\lambda^{\frac{n}{2}+\frac{n}{4}\beta(m_{T_\eta})}
\big(1+\bigO(\varepsilon)+\mathrm{o}(1)\big) + \bigO(\lambda^{\frac{n}{2}+\frac{n}{4}\bar\beta}) .
$$

Finally, Lemma \ref{lem:freezing-simplex} allows us to replace $\beta(m_{T_\eta})$ by $\beta(m)$ inside each simplex at the price of a relative $1+\bigO(\varepsilon)$ error, and $\varepsilon$-quasi-isometry gives $\vol_{h_\eta}(T_\eta)=(1+\bigO(\varepsilon))\vol_{h_0}(T_\eta)$. 
Moreover, because $\beta$ is Lipschitz and the simplices have diameter $\mathrm{O}(\eta)$, the union of simplices with $\beta(m_{T_\eta})\geq\bar\beta$ is contained in $\{\beta\geq\bar\beta-C\eta\}$ and contains $\{\beta\geq\bar\beta+C\eta\}$ for some $C>0$. Since $\bar\beta>\beta_c$ is fixed, the coefficient $A(\beta,n)$ is uniformly bounded on the strip $\{\bar\beta-C\eta\leq \beta\leq \bar\beta+C\eta\}$ for $\lambda$ large enough. Moreover, on this strip one has
$\lambda^{\frac{n}{2}+\frac{n}{4}\beta(m)} \leq \lambda^{\frac{n}{2}+\frac{n}{4}\bar\beta}\exp\big(\frac{nC}{4}\eta\ln\lambda\big) = \mathrm{O}\big(\lambda^{\frac{n}{2}+\frac{n}{4}\bar\beta}\big)$,
because $\eta\ln\lambda\to 0$. Summing over the simplices under consideration, whose total $h_0$-volume is uniformly bounded, we obtain that the contribution of the strip is $\mathrm{O}\big(\lambda^{\frac{n}{2}+\frac{n}{4}\bar\beta}\big)$, hence lower order.
Therefore the main sum is a Riemann sum for $\int_{\{\beta\geq\bar\beta\}}A(\beta(m),n)\,\lambda^{\frac{n}{2}+\frac{n}{4}\beta(m)}\,d\vol_{h_0}(m)$, up to a relative $\mathrm{o}(1)$ error.
Since the same main term is obtained for both Dirichlet and Neumann boundary conditions, the initial bracketing gives the desired asymptotic up to a relative error $\mathrm{O}(\varepsilon)+\mathrm{o}(1)$. Since $\varepsilon>0$ was arbitrary, letting $\varepsilon\to 0$ concludes the proof.

\end{proof}

\subsection{The case $\beta \leq \beta_c $}\label{sec:phase-}

\begin{proof}[Proof of the second item of Theorem \ref{thm:main}.]
Assume $\bm\leq\beta_c$. As in the supercritical case, by Theorem \ref{theo:normal}, Proposition \ref{prop:normal-quasi} and quasi-isometry stability (Appendix \ref{app:quasi}), we work with the model metric $g_0=dx^2+x^{-\beta(m)}h_0$ on a collar $X_a$. Choose $c\in(0,\pi/2)$ as in Lemma \ref{le:0ev} and fix $b\in(0,a)$.

Dirichlet-Neumann bracketing and Lemma \ref{le:0ev} yield, for $\lambda$ large enough,
$$
N_{X\setminus X_{c\lambda^{-1/2}}}^D(\lambda)\leq N_X(\lambda)\leq N_{X\setminus X_{c\lambda^{-1/2}}}^N(\lambda),
$$
where the superscripts indicate Dirichlet or Neumann boundary conditions imposed on the artificial boundary $\{x=c\lambda^{-1/2}\}$.
Applying bracketing again to the decomposition $X\setminus X_{c\lambda^{-1/2}}=(X\setminus X_b)\cup(X_b\setminus X_{c\lambda^{-1/2}})$ across the artificial boundary $\{x=b\}$ yields
$$
N_{X\setminus X_b}^D(\lambda)+N_{X_b\setminus X_{c\lambda^{-1/2}}}^D(\lambda)\leq N_X(\lambda)\leq N_{X\setminus X_b}^N(\lambda)+N_{X_b\setminus X_{c\lambda^{-1/2}}}^N(\lambda).
$$
The region $X\setminus X_b$ is smooth with finite volume, hence by the classical Weyl law, for Dirichlet and Neumann boundary conditions on $\partial(X\setminus X_b)$,
$$
N_{X\setminus X_b}^{D/N}(\lambda) = \gamma_{n+1}\,\lambda^{\frac{n+1}{2}}\,\vol_g(X\setminus X_b)+\mathrm{o}(\lambda^{\frac{n+1}{2}}).
$$
It remains to estimate the collar contribution $N_{X_b\setminus X_{c\lambda^{-1/2}}}^{D/N}(\lambda)$.

Fix $\varepsilon\in(0,1)$ and choose $\gamma\in(0,\frac{1}{2})$. Set $\eta=\lambda^{\gamma-\frac{1}{2}}$ and let $\mathcal T_\eta$ and $h_\eta$ be given by Proposition \ref{prop:triang} applied to $h_0$. Choose $C_\varepsilon>0$ large enough, depending only on $\varepsilon$ and on the finite set of shapes, and set $\zeta=C_\varepsilon\lambda^{-1/2}$. Decompose $[c\lambda^{-1/2},b]$ into intervals
$$
I_k=[c\lambda^{-1/2}+k\zeta,c\lambda^{-1/2}+(k+1)\zeta],\qquad k=0,\dots,N_\lambda,
$$
where $N_\lambda$ is the largest integer such that $c\lambda^{-1/2}+N_\lambda\zeta\leq b<c\lambda^{-1/2}+(N_\lambda+1)\zeta$.
Then $N_\lambda\simeq b/\zeta$ as $\lambda\to+\infty$. For each simplex $T_\eta\in\mathcal T_\eta$ and each $k$, set $B_{k,T_\eta}=I_k\times T_\eta$.

On $B_{k,T_\eta}$, freeze the coefficients by picking $m_{T_\eta}\in T_\eta$ and $x_k=c\lambda^{-1/2}+k\zeta$, and consider the constant coefficient metric
$$
g_{k,T_\eta}=dx^2+x_k^{-\beta(m_{T_\eta})}h_\eta\vert_{T_\eta}.
$$
Since $\beta\in C^1(M)$ and $h_\eta$ is $\varepsilon$-quasi-isometric to $h_0$, for $\lambda$ large the metrics $g_0$ and $g_{k,T_\eta}$ are $\varepsilon$-quasi-isometric on $B_{k,T_\eta}$, uniformly in $k$ and $T_\eta$.

By an affine chart sending $T_\eta$ to a simplex in $\R^n$ and by the rescaling
$$
x=x_k+\zeta t,\qquad y=\zeta x_k^{\beta(m_{T_\eta})/2}z,
$$
the box $(B_{k,T_\eta},g_{k,T_\eta})$ becomes isometric, after rescaling lengths by the factor $\zeta$ (equivalently, after factoring out $\zeta^2$ from the metric), to the Euclidean domain
$$
D_{k,T_\eta}=[0,1]\times\left(\frac{\eta}{\zeta}x_k^{-\beta(m_{T_\eta})/2}\right)T_0\subset\R^{n+1},
$$
and the spectral parameter rescales to $\mu=\lambda\zeta^2=C_\varepsilon^2$.
Since $\eta/\zeta=\lambda^\gamma/C_\varepsilon\to+\infty$, the tangential scaling factor in $D_{k,T_\eta}$ tends to $+\infty$ uniformly in $k$ and $T_\eta$. We subdivide the tangential simplex in $D_{k,T_\eta}$ into simplices of diameter at most $1$ using the Freudenthal subdivision. This yields a disjoint union of cells whose shapes belong to a finite family independent of $k$, $T_\eta$ and $\lambda$. Dirichlet-Neumann bracketing across the internal faces and Lemma \ref{le:uniform} (applied with $d=n+1$ and $\mu=C_\varepsilon^2$) give
$$
N_{B_{k,T_\eta}}^{D/N}(\lambda)=\gamma_{n+1}\vol_{g_{k,T_\eta}}(B_{k,T_\eta})\lambda^{\frac{n+1}{2}}(1+\bigO(\varepsilon)),
$$
uniformly in $k$ and $T_\eta$.

Summing over $k$ and $T_\eta\in\mathcal T_\eta$ and using once more quasi-isometry and Dirichlet-Neumann bracketing yields
$$
N_{X_b\setminus X_{c\lambda^{-1/2}}}^{D/N}(\lambda) = \gamma_{n+1} \, \lambda^{\frac{n+1}{2}} \, \vol_g(X_b\setminus X_{c\lambda^{-1/2}}) + \bigO(\varepsilon\lambda^{\frac{n+1}{2}}).
$$
Combining with the interior contribution gives
$$
N_X(\lambda) = \gamma_{n+1} \, \lambda^{\frac{n+1}{2}} \, \vol_g(X\setminus X_{c\lambda^{-1/2}}) + \bigO(\varepsilon\lambda^{\frac{n+1}{2}}) + \mathrm{o}(\lambda^{\frac{n+1}{2}}).
$$
Finally, Lemma \ref{lemm:uniform} below shows that replacing $c\lambda^{-1/2}$ by $\lambda^{-1/2}$ in the truncated volume only affects lower order terms when $\bm\leq\beta_c$. Letting $\varepsilon\to0$ proves the claim.
\end{proof}

%

\begin{lemma}\label{lemm:uniform}
Given $\gamma\in[0,1]$ and $0<c<C$, let
$$
Q(\lambda)=\frac{\int_c^C \frac{dx}{x^{1-\gamma}}}{\int_c^{\sqrt{\lambda}} \frac{dx}{x^{1-\gamma}}}.
$$
Then $Q(\lambda)\to 0$ uniformly in $\gamma$ as $\lambda\to+\infty$. 
\end{lemma}

\begin{proof}
If $\gamma=0$, then $Q(\lambda)=\frac{\ln(C/c)}{\ln(\sqrt{\lambda}/c)}\to 0$ as $\lambda\to+\infty$. 
Assume now $\gamma\in(0,1]$. We compute
$$
Q(\lambda)=\frac{(C/c)^\gamma-1}{(\sqrt{\lambda}/c)^\gamma-1}.
$$
Write $t=\gamma\ln(\sqrt{\lambda}/c)\geq 0$. Then $e^t-1\geq t$, hence $(\sqrt{\lambda}/c)^\gamma-1 \ge \gamma\ln(\sqrt{\lambda}/c)$. Similarly,
$$
(C/c)^\gamma-1 = e^{\gamma\ln(C/c)}-1 \leq \gamma\ln(C/c)\,e^{\gamma\ln(C/c)} \leq \gamma\ln(C/c)\,(C/c).
$$
Therefore, $Q(\lambda)\leq \frac{(C/c)\ln(C/c)}{\ln(\sqrt{\lambda}/c)}\to 0$ uniformly in $\gamma\in(0,1]$.
\end{proof}

\section{Proof of Corollary \ref{co:codimension}}\label{proof_cor_co:codimension}

\begin{proof}
1. Assume $\bm>\beta_c$. Fix $\bar\beta\in(\beta_c,\bm)$. By Theorem \ref{thm:main}, item 1,
$$
N(\lambda)\sim\int_{\{\beta\geq\bar\beta\}}A(\beta(m),n)\lambda^{\frac{n}{2}+\frac{n}{4}\beta(m)}\,d\vol_{h_0}(m).
$$
Factorizing the maximum exponent yields
$$
N(\lambda)\sim \lambda^{\frac{n}{2}+\frac{n}{4}\bm}
\int_{\{\beta\geq\bar\beta\}}A(\beta(m),n)\exp\left(\tau(\beta(m)-\bm)\right)\,d\vol_{h_0}(m),
\qquad
\tau=\frac{n}{4}\ln\lambda.
$$
Since the Laplace integral localizes near $W$ as $\tau\to+\infty$ and $A(\beta,n)$ is continuous on $[\bar\beta,\bm]$, we may replace $A(\beta(m),n)$ by $A(\bm,n)$ up to a relative $\mathrm{o}(1)$ error inside the integral.
Applying the Morse-Bott Laplace method (Proposition \ref{prop:gauss} of Appendix \ref{app:gauss}) to $f=\beta$ and $a\equiv 1$ gives
$$
\int_{\{\beta\geq\bar\beta\}}\exp\left(\tau(\beta(m)-\bm)\right)\,d\vol_{h_0}(m)
\sim (2\pi)^{\frac{d}{2}}\tau^{-\frac{d}{2}}
\int_W \frac{d\vol_{h_0\vert_W}(w)}{\sqrt{\det Q(w)}}.
$$
Since $\tau^{-\frac{d}{2}}=\big(\frac{n}{4}\big)^{-\frac{d}{2}}(\ln\lambda)^{-\frac{d}{2}}$, the first claim follows.

\medskip
\noindent
2. Assume $\bm=\beta_c=\frac{2}{n}$. By Theorem \ref{thm:main}, item 2,
$$
N(\lambda)\sim\gamma_{n+1}\lambda^{\frac{n+1}{2}}\vol_g(\{\dist(\cdot,M)>\lambda^{-1/2}\}).
$$
By Theorem \ref{theo:normal} and the quasi-isometry stability in Appendix \ref{app:quasi}, the truncated volume can be computed on the model metric $g_0=dx^2+x^{-\beta(m)}h_0$ up to a relative $\mathrm{o}(1)$ error. In a collar neighborhood $(0,b)\times M$, we have $d\vol_{g_0}=x^{-\frac{n}{2}\beta(m)}\,dx\,d\vol_{h_0}(m)$.
Hence
$$
\vol_{g_0}(\{\dist(\cdot,M)>\lambda^{-1/2}\})
=
\int_{\lambda^{-1/2}}^b\int_M x^{-\frac{n}{2}\beta(m)}\,d\vol_{h_0}(m)\,dx+\mathrm{O}(1).
$$
Since $\bm=\frac{2}{n}$, we write
$x^{-\frac{n}{2}\beta(m)}=x^{-1}\exp\left(-\frac{n}{2}(\bm-\beta(m))\vert\ln x\vert\right)$.
Near $W$, choose $h_0$-normal coordinates $(w,z)\in W\times\R^d$ so that
$$
\beta(w,z)=\bm-\frac{1}{2}\langle Q(w)z,z\rangle+\mathrm{O}(\Vert z\Vert^3).
$$
Setting $t=\vert\ln x\vert$, we obtain uniformly for large $t$,
$$
\int_M x^{-\frac{n}{2}\beta(m)}\,d\vol_{h_0}(m) = x^{-1}(2\pi)^{\frac{d}{2}}\Big(\frac{n}{2}\Big)^{-\frac{d}{2}}t^{-\frac{d}{2}} \int_W \frac{d\vol_{h_0\vert_W}(w)}{\sqrt{\det Q(w)}}+\mathrm{O}\big(x^{-1}t^{-\frac{d+1}{2}}\big),
$$
by Gaussian integration in the normal directions. Therefore,
$$
\vol_{g_0}(\{\dist(\cdot,M)>\lambda^{-1/2}\}) = (2\pi)^{\frac{d}{2}}\Big(\frac{n}{2}\Big)^{-\frac{d}{2}} \int_W \frac{d\vol_{h_0\vert_W}(w)}{\sqrt{\det Q(w)}} \int_{\vert\ln b\vert}^{\frac{1}{2}\ln\lambda+\mathrm{O}(1)} t^{-\frac{d}{2}}\,dt+\mathrm{O}(1).
$$
If $d=1$, the $t$-integral is asymptotic to $\sqrt{2}\sqrt{\ln\lambda}$, hence
$$
\vol_{g_0}(\{\dist(\cdot,M)>\lambda^{-1/2}\}) \sim 2\sqrt{\frac{2\pi}{n}}\sqrt{\ln\lambda}\int_W \frac{d\vol_{h_0\vert_W}(w)}{\sqrt{\det Q(w)}}.
$$
If $d=2$, the $t$-integral is asymptotic to $\ln(\ln\lambda)$, hence
$$
\vol_{g_0}(\{\dist(\cdot,M)>\lambda^{-1/2}\})
\sim \frac{4\pi}{n}\ln(\ln\lambda)\int_W \frac{d\vol_{h_0\vert_W}(w)}{\sqrt{\det Q(w)}}.
$$
If $d>2$, the $t$-integral converges as $\lambda\to+\infty$, so $\vol_g(X)<+\infty$ and the third claim follows.
Multiplying by $\gamma_{n+1}\lambda^{\frac{n+1}{2}}$ concludes the proof of item 2.

\medskip
\noindent
3. Assume $\bm<\beta_c=\frac{2}{n}$. Then $x\mapsto x^{-\frac{n}{2}\beta(m)}$ is integrable near $x=0$ uniformly in $m$, hence $\vol_g(X)<+\infty$. By Theorem \ref{thm:main}, item 2,
$$
N(\lambda)\sim\gamma_{n+1}\lambda^{\frac{n+1}{2}}\vol_g(\{\dist(\cdot,M)>\lambda^{-1/2}\}).
$$
Since $\vol_g(\{\dist(\cdot,M)>\lambda^{-1/2}\})\to\vol_g(X)$ as $\lambda\to+\infty$, the claimed asymptotic follows.
\end{proof}

\appendix

\section{Appendix}

\subsection{Quasi-isometries}\label{app:quasi}

Let $Y$ be a compact manifold (possibly with boundary) of dimension $d$ and let $g_1,g_2$ be two Riemannian metrics on $Y$. 

\begin{lemma}\label{lemm:quasi}
Assume that $K^{-1}g_1\leq g_2\leq Kg_1$ as quadratic forms. Denote by $\lambda_j(g_i)$ the eigenvalues of the Friedrichs Laplacian $\triangle_{g_i}$ (Dirichlet or Neumann boundary conditions on $\partial Y$ if $\partial Y\ne\emptyset$). Then 
$$
K^{-1-d}\lambda_j(g_1)\leq\lambda_j(g_2)\leq K^{1+d}\lambda_j(g_1) \qquad\forall j\geq 1.
$$
\end{lemma}

\begin{proof}
For any $u\in H^1(Y)$ (with the relevant boundary condition), we have
$\Vert\nabla u\Vert_{g_2}^2\leq K\Vert\nabla u\Vert_{g_1}^2$ and $d\vol_{g_2}\leq K^{\frac{d}{2}}\,d\vol_{g_1}$, hence
$$
\int_Y\Vert\nabla u\Vert_{g_2}^2\,d\vol_{g_2}\leq K^{1+\frac{d}{2}}\int_Y\Vert\nabla u\Vert_{g_1}^2\,d\vol_{g_1}.
$$
Moreover $d\vol_{g_2}\geq K^{-\frac{d}{2}}d\vol_{g_1}$, hence
$$
\int_Y u^2\,d\vol_{g_2}\geq K^{-\frac{d}{2}}\int_Y u^2\,d\vol_{g_1}.
$$
Therefore the Rayleigh quotients satisfy
$$
\frac{\int_Y\Vert\nabla u\Vert_{g_2}^2\,d\vol_{g_2}}{\int_Y u^2\,d\vol_{g_2}}
\leq
K^{1+d}\frac{\int_Y\Vert\nabla u\Vert_{g_1}^2\,d\vol_{g_1}}{\int_Y u^2\,d\vol_{g_1}},
$$
and similarly with $g_1$ and $g_2$ exchanged. The lemma follows from the min-max principle.
\end{proof}

\begin{definition}\label{defi:quasi}  
Given $\varepsilon>0$, we say that two metrics $g_1$ and $g_2$ are $\varepsilon$-quasi-isometric if $K^{-1}g_1\leq g_2\leq Kg_1$ as quadratic forms with $K=1+\varepsilon$.
\end{definition}

\begin{lemma}\label{lem:counting-quasi}
There exists a constant $C>0$ depending only on $d=\dim(Y)$ such that the following holds. If $g_1$ and $g_2$ are $(1+\varepsilon)$-quasi-isometric with $\varepsilon\in(0,1)$, then
$$
N_{g_1}\left(\frac{\lambda}{1+C\varepsilon}\right)\leq N_{g_2}(\lambda)\leq N_{g_1}\big((1+C\varepsilon)\lambda\big)
\qquad \forall \lambda>0.
$$
\end{lemma}

\begin{proof}
By Lemma \ref{lemm:quasi}, there exists $C>0$ (depending on $d$) such that
$$
\frac{1}{1+C\varepsilon}\lambda_j(g_1)\leq \lambda_j(g_2)\leq (1+C\varepsilon)\lambda_j(g_1)
\qquad \forall j\geq 1,
$$
because $(1+\varepsilon)^{1+d}=1+\mathrm{O}(\varepsilon)$ and $(1+\varepsilon)^{-1-d}=1+\mathrm{O}(\varepsilon)$ for $\varepsilon\in(0,1)$.
The inequalities for counting functions follow by monotonicity.
\end{proof}

It follows that if $N_{g_1}$ has a ``polynomial'' growth and we can choose $\varepsilon>0$ as small as we want, then the asymptotics are the same.

\subsection{Truncation for one-dimensional Schr\"odinger operators and constant-exponent cells}\label{app:truncation}

\begin{proposition}\label{prop:1d-truncation}
Let $V\in C^\infty(\R^+)$ satisfy $V\geq 0$ and $V(x)\to+\infty$ as $x\to+\infty$. Let $P$ be the nonnegative selfadjoint realization of $-\partial_x^2+V$ on $L^2(\R^+)$, with a fixed selfadjoint boundary condition at $x=0$, and denote by $(\lambda_j)_{j\geq 1}$ its eigenvalues. For $L>0$, let $P_L^{D/N}$ be the realization of the same differential operator on $(0,L)$, with the same boundary condition at $x=0$ and with Dirichlet or Neumann boundary conditions at $x=L$. Denote by $\lambda_j^{D/N}(L)$ the corresponding eigenvalues. Then, for every $j\geq 1$, $\lambda_j^{D/N}(L)\rightarrow \lambda_j$ as $L\to+\infty$. 

If, in addition, $V(x)\to+\infty$ as $x\to 0^+$ and $P$ is the Friedrichs realization of $-\partial_x^2+V$ on $L^2(\R^+)$, then the eigenvalues $\lambda_j^{D/N}(a,L)$ of $-\partial_x^2+V$ on $(a,L)$ with Dirichlet or Neumann boundary conditions at both endpoints satisfy $\lambda_j^{D/N}(a,L)\rightarrow\lambda_j$ as $a\to 0^+$ and $L\to+\infty$.
\end{proposition}

\begin{proof}
The Dirichlet convergence follows from the min-max principle. Indeed, the form domains on $(0,L)$, with the fixed boundary condition at $x=0$ and Dirichlet condition at $x=L$, form an increasing family as $L\to+\infty$ and their union is a form core for $P$. Hence $\lambda_j^D(L)$ is nonincreasing in $L$, bounded below by $\lambda_j$, and converges to $\lambda_j$.
Consider now the Neumann case. Let $\varphi_{j,L}$ be an $L^2$-normalized eigenfunction of $P_L^N$ associated with $\lambda_j^N(L)$. By min-max, $\lambda_j^N(L)\leq\lambda_j^D(L)$, so $(\lambda_j^N(L))_{L\geq L_0}$ is uniformly bounded. We extend $\varphi_{j,L}$ by $0$ outside $(0,L)$ only when discussing $L^2$ convergence.
Fix $R>0$. Because $V\to+\infty$ at $+\infty$, there exists $M_R>0$ such that $V\geq M_R$ on $(R,+\infty)$. Therefore, for $L>R$,
$$
M_R\int_R^{L}\vert\varphi_{j,L}(x)\vert^2\,dx \leq \int_0^{L}V(x)\vert\varphi_{j,L}(x)\vert^2\,dx \leq \lambda_j^N(L),
$$
and thus the $L^2$ mass of $\varphi_{j,L}$ on $(R,L)$ can be made arbitrarily small, uniformly for large $L$, by taking $R$ large enough. On every fixed interval $(0,R)$ the family $(\varphi_{j,L})$ is uniformly bounded in $H^1$ for $L>R$, since
$$
\int_0^R\vert\varphi_{j,L}'(x)\vert^2\,dx\leq \int_0^{L} \left(\vert\varphi_{j,L}'(x)\vert^2+V(x)\vert\varphi_{j,L}(x)\vert^2\right) dx=\lambda_j^N(L).
$$
By Rellich's theorem, any sequence $L_p\to+\infty$ admits a subsequence such that $\varphi_{j,L_p}$ converges strongly in $L^2(\R^+)$ to some nonzero limit $\varphi_j$. Passing to the limit in the weak formulation shows that $P\varphi_j=\mu_j\varphi_j$, where $\mu_j$ is the limit of the corresponding eigenvalues. Hence $\mu_j$ is an eigenvalue of $P$. The identification $\mu_j=\lambda_j$ follows by induction on $j$, using orthogonality of the eigenfunctions and the min-max principle. 

When $V$ also tends to $+\infty$ at $x=0$, the same compactness argument applies symmetrically at the left endpoint: extending $\varphi_{j,a,L}$ by zero on $(0,a)\cup(L,+\infty)$, the bound $V\geq M$ on $(0,a_0)\cup(R,+\infty)$ for $a_0$ small and $R$ large yields uniform smallness of the $L^2$ mass outside $[a_0,R]$, and the same Rellich extraction concludes.
\end{proof}

\begin{remark}
The same compactness argument applies to Schr\"odinger operators $-\triangle+V$ on exhausting domains in any dimension, under the same confining assumption $V\to+\infty$ near the boundary of the ambient domain. We do not need this generalization here.
\end{remark}

The convergence in Proposition \ref{prop:1d-truncation} holds for each fixed $k$, but it does not by itself control the spectral sum over $k$. We need the following uniform Weyl upper bound, valid on \emph{every} interval and for both boundary conditions.

\begin{lemma}\label{lem:uniform-1d-Neumann}
Fix $\bar\beta>\beta_c$ and $B\geq\bar\beta$. There exist constants $C,E_0>0$, depending only on $n,\bar\beta,B$, such that for all $\beta\in[\bar\beta,B]$, every interval $(a,L)$ with $0\leq a<L\leq+\infty$ (with the Friedrichs realization at $a=0$ when $a=0$), and every $E\geq E_0$, the Neumann realization of
$P_1(\beta)=-\partial_s^2+\frac{\beta n(\beta n+4)}{16s^2}+s^\beta$
on $(a,L)$ satisfies
$$
N_{P_1(\beta),(a,L)}^N(E)\leq C\,E^{\frac{\beta+2}{2\beta}}.
$$
Equivalently, its eigenvalues satisfy $\nu_k^N(\beta,a,L)\geq C^{-1}k^{\frac{2\beta}{\beta+2}}$ for every $k\geq 1$. The same bounds hold a fortiori for the Dirichlet realization.
\end{lemma}

\begin{proof}
It is enough to prove the Neumann estimate, since the Dirichlet counting function is smaller by the min-max principle. We give a direct proof on an arbitrary interval $(a,L)$. Drop the nonnegative inverse-square term; by the min-max principle,
$$
N_{P_1(\beta),(a,L)}^N(E)\leq N_{-\partial_s^2+s^\beta,(a,L)}^N(E).
$$
Partition $(a,L)$ by the unit intervals $I_p=(p,p+1)$, $p\geq 0$, and set $J_p=(a,L)\cap I_p$. Imposing Neumann conditions at the artificial interfaces, Neumann bracketing gives
$$
N_{-\partial_s^2+s^\beta,(a,L)}^N(E)\leq\sum_{p\geq 0}N_{J_p}^N\big(E-p^\beta\big),
$$
since $s^\beta\geq p^\beta$ on $J_p$. Each nonempty $J_p$ has length at most $1$, so the Neumann counting function of $-\partial_s^2$ on $J_p$ satisfies
$N_{J_p}^N(\mu)\leq 1+\frac{1}{\pi}\sqrt{\mu_+}$
for $\mu\geq 0$, while it equals $0$ for $\mu<0$. Hence only the indices $p\leq E^{1/\beta}$ contribute, and
$$
N_{P_1(\beta),(a,L)}^N(E)\leq\sum_{0\leq p\leq E^{1/\beta}}\Big(1+\frac{1}{\pi}\sqrt{E-p^\beta}\Big)\leq C_1 E^{1/\beta}+C_2\int_0^{E^{1/\beta}}\sqrt{E-t^\beta}\,dt.
$$
We have
$\int_0^{E^{1/\beta}}\sqrt{E-t^\beta}\,dt = E^{\frac{1}{\beta}+\frac12}\int_0^1\sqrt{1-u^\beta}\,du$,
and $\int_0^1\sqrt{1-u^\beta}\,du$ is bounded uniformly for $\beta\in[\bar\beta,B]$. Since $\frac1\beta+\frac12=\frac{\beta+2}{2\beta}\geq\frac1\beta$, we obtain
$N_{P_1(\beta),(a,L)}^N(E)\leq C\,E^{\frac{\beta+2}{2\beta}}$
for $E\geq E_0$, with $C,E_0$ depending only on $n,\bar\beta,B$ and uniform in the interval $(a,L)$. The equivalent eigenvalue bound follows by inversion, after increasing $C$ if necessary to handle the finitely many small values of $k$, and the Dirichlet case is smaller by min-max.
\end{proof}

We now give the proof of Lemma \ref{lem_constant-cell}.

\begin{proof}[Proof of Lemma \ref{lem_constant-cell}.]
Let $\widehat\Omega_T=(0,+\infty)\times T$ with metric $dx^2+x^{-\beta}\vert dy\vert^2$, and impose the same Dirichlet or Neumann boundary conditions on $(0,+\infty)\times\partial T$. Let $(\mu_j^{D/N}(T))_{j\geq 1}$ be the corresponding eigenvalues of $-\triangle_T$ on $T$. 
In the Neumann case, the tangential zero mode has multiplicity one, since $T$ is connected. On the finite and truncated collars $\Omega_T$ and $\Omega_{T,\lambda,c}$, this zero mode gives only a one-dimensional radial contribution, hence an $\mathrm{O}(\lambda^{1/2})$ contribution, uniformly in the mesoscopic simplexes considered here. Since $\frac{1}{2}<\frac{n}{2}+\frac{n}{4}\beta$ for $n\geq 1$ and $\beta>0$, this contribution is negligible. On the infinite cone $\widehat\Omega_T$, the zero tangential mode does not belong to the discrete supercritical model, since the corresponding radial operator $-\partial_x^2+\frac{\beta n(\beta n+4)}{16x^2}$ has no confining term $\mu x^\beta$. We therefore restrict the separation-of-variables argument to the positive tangential eigenvalues, and from now on write $\mu_j(T)>0$. Let $N_{T,+}^{D/N}(\Lambda)$ denote the counting function of the positive tangential eigenvalues on $T$.
By separation of variables, the positive-tangential part of the Friedrichs Laplacian on $\widehat\Omega_T$ is unitarily equivalent to the direct sum over the positive tangential eigenvalues of the one-dimensional operators
$P_{\mu_j(T)}=-\partial_x^2+\frac{\beta n(\beta n+4)}{16x^2}+\mu_j(T)x^\beta$
on $L^2(\R^+)$. After the scaling $x=\mu_j(T)^{-1/(\beta+2)}s$, this becomes
$\mu_j(T)^{\frac{2}{\beta+2}}\big(-\partial_s^2+\frac{\beta n(\beta n+4)}{16s^2}+s^\beta\big)$.
If $(\nu_k(\beta))_{k\geq 1}$ denotes the spectrum of $P_1=-\partial_s^2+\frac{\beta n(\beta n+4)}{16s^2}+s^\beta$ on $L^2(\R^+)$, then
$$
N_{\widehat\Omega_T,+}^{D/N}(\lambda) = \sum_{k\geq 1} N_{T,+}^{D/N} \bigg(\Big(\frac{\lambda}{\nu_k(\beta)}\Big)^{\frac{\beta+2}{2}}\bigg).
$$
Since only finitely many shapes occur, the Weyl law on $T$ is uniform with respect to $T$, and by scaling the dependence on the size of $T$ is exactly through $\vert T\vert$. Since $\beta$ ranges in the compact interval $[\bar\beta,\bm]\subset(\beta_c,+\infty)$, the series defining $A(\beta,n)$ converges uniformly. Therefore the same separation-of-variables argument as in \cite{CDHDT24} yields
$$
N_{\widehat\Omega_T,+}^{D/N}(\lambda) = A(\beta,n)\,\vert T\vert\,\lambda^{\frac{n}{2}+\frac{n}{4}\beta} \,(1+\mathrm{o}(1))
$$
as $\lambda\to+\infty$, uniformly with respect to $\beta\in[\bar\beta,\bm]$ and $T$.

We next pass from the infinite cone $\widehat\Omega_T$ to the finite collar $\Omega_T=(0,b)\times T$. The separation of variables performed above applies verbatim on $\Omega_T$ with Dirichlet or Neumann boundary conditions on $\{x=b\}\cup((0,b)\times\partial T)$, the only difference being that each radial operator $P_{\mu_j(T)}$ is now realized on the bounded interval $(0,b)$ with the same selfadjoint condition at $x=0$ and Dirichlet or Neumann condition at $x=b$. The scaling $x=\mu_j(T)^{-1/(\beta+2)}s$ maps each such radial operator to $\mu_j(T)^{2/(\beta+2)}P_1$ acting on $L^2(0,L_j)$, with $L_j=b\mu_j(T)^{1/(\beta+2)}\to+\infty$ as $j\to+\infty$. Denoting by $\nu_k^{D/N}(\beta,L)$ the $k$-th eigenvalue of $P_1$ on $L^2(0,L)$ (with the same selfadjoint condition at $s=0$ and Dirichlet or Neumann condition at $s=L$), the eigenvalues of the Friedrichs Laplacian on $\Omega_T$ are therefore
$$
\mu_j(T)^{\frac{2}{\beta+2}}\,\nu_k^{D/N}(\beta,L_j),\qquad j,k\geq 1.
$$
Proposition \ref{prop:1d-truncation} applied to $P_1$ gives, for each $k\geq 1$ and each $\beta\in[\bar\beta,\bm]$, that $\nu_k^{D/N}(\beta,L)\rightarrow\nu_k(\beta)$ as $L\to+\infty$. 
The compactness proof of Proposition \ref{prop:1d-truncation} is uniform for $\beta\in[\bar\beta,\bm]$, because the potentials $\frac{\beta n(\beta n+4)}{16s^2}+s^\beta$ are uniformly confining at $0$ and at $+\infty$ on this compact interval of exponents. Hence the convergence is uniform with respect to $\beta\in[\bar\beta,\bm]$.
The same separation of variables on the finite collar $\Omega_T$ gives
$$
N_{\Omega_T}^{D/N}(\lambda)
=
\sum_{j,k}\mathbf 1_{\big\{\mu_j(T)^{\frac{2}{\beta+2}}\nu_k^{D/N}(\beta,L_j)\leq\lambda\big\}}
+R_{0,T}^{D/N}(\lambda),
\qquad L_j=b\,\mu_j(T)^{\frac{1}{\beta+2}},
$$
where the sum is over the positive tangential eigenvalues $\mu_j(T)>0$ and $k\geq 1$. Here $R_{0,T}^D(\lambda)=0$, while in the Neumann case $R_{0,T}^N(\lambda)=\bigO(\lambda^{1/2})$, uniformly in the mesoscopic simplexes considered here, since it is the contribution of the tangential zero mode discussed above; it is negligible compared with $\lambda^{\frac{n}{2}+\frac{n}{4}\beta}$. Note that $L_j$ depends on the tangential index $j$, so the sum cannot be regrouped as a tangential counting function with a single radial threshold.

Since $T$ is mesoscopic with diameter comparable to $\eta=\lambda^{\gamma-\frac12}$, the first positive tangential eigenvalue satisfies $\mu_1(T)\geq C\eta^{-2}$, whence
$L_j\geq Cb\,\eta^{-\frac{2}{\beta+2}}\to+\infty$
uniformly in $j$, $T$ and $\beta\in[\bar\beta,\bm]$. We work with the normalized quantity $N_{\Omega_T}^{D/N}(\lambda)/\lambda^{\frac{n}{2}+\frac{n}{4}\beta}$ and fix a cutoff $K\geq 1$.

For the head $k\leq K$, Proposition \ref{prop:1d-truncation} gives $\nu_k^{D/N}(\beta,L_j)=\nu_k(\beta)(1+\mathrm{o}(1))$ as $\lambda\to+\infty$, uniformly in $j$, $T$ and $\beta\in[\bar\beta,\bm]$ (the convergence is uniform because the potentials are uniformly confining on the compact exponent interval). Using the uniform Weyl law for the positive-tangential counting function on the finite family of simplex shapes and $\frac{n(\beta+2)}{4}=\frac{n}{2}+\frac{n}{4}\beta=\frac{d_H}{2}$, the contribution of these finitely many indices, normalized by $\lambda^{\frac{n}{2}+\frac{n}{4}\beta}$, converges as $\lambda\to+\infty$ to $\gamma_n\vert T\vert\sum_{k=1}^K\nu_k(\beta)^{-d_H/2}$.

For the tail $k>K$, Lemma \ref{lem:uniform-1d-Neumann} gives $\nu_k^{D/N}(\beta,L_j)\geq C^{-1}k^{2\beta/(\beta+2)}$, uniformly in $j$, $T$ and $\beta\in[\bar\beta,\bm]$. Hence, if a pair $(j,k)$ contributes, then $\mu_j(T)^{2/(\beta+2)}\leq\lambda\,C\,k^{-2\beta/(\beta+2)}$, that is,
$$
\mu_j(T)\leq C^{\frac{\beta+2}{2}}\lambda^{\frac{\beta+2}{2}}k^{-\beta}.
$$
By the uniform Weyl upper bound for the positive-tangential counting function, $\#\{j:\mu_j(T)\leq\Lambda\}\leq C'\vert T\vert\,\Lambda^{n/2}$, so
$$
\sum_{k>K}\#\big\{j:\mu_j(T)^{\frac{2}{\beta+2}}\nu_k^{D/N}(\beta,L_j)\leq\lambda\big\}
\leq C''\vert T\vert\,\lambda^{\frac{n}{2}+\frac{n}{4}\beta}\sum_{k>K}k^{-\frac{n\beta}{2}}.
$$
After normalization by $\lambda^{\frac{n}{2}+\frac{n}{4}\beta}$, the factor $\lambda^{\frac{n}{2}+\frac{n}{4}\beta}$ cancels, leaving the bound $C''\vert T\vert\sum_{k>K}k^{-n\beta/2}$, which is independent of $\lambda$. The series $\sum_k k^{-n\beta/2}$ converges because $\beta>\beta_c=\frac{2}{n}$; this is the same condition that makes $\zeta_{\beta,n}(d_H/2)$, and hence $A(\beta,n)$ in Theorem \ref{theo:main0}, finite. Therefore the normalized tail tends to $0$ as $K\to+\infty$, uniformly in $\lambda$ and uniformly for $\beta\in[\bar\beta,\bm]$ (since $n\bar\beta/2>1$).

Letting first $\lambda\to+\infty$ and then $K\to+\infty$ gives
$$
N_{\Omega_T}^{D/N}(\lambda)=A(\beta,n)\,\vert T\vert\,\lambda^{\frac{n}{2}+\frac{n}{4}\beta}\,(1+\mathrm{o}(1)),
$$
that is, $N_{\Omega_T}^{D/N}(\lambda)$ has the same leading term as $N_{\widehat\Omega_T,+}^{D/N}(\lambda)$.

Finally, let $\Omega_{T,\lambda,c}=[c\lambda^{-1/2},b]\times T$. We redo the scaling. The same separation of variables applies on $\Omega_{T,\lambda,c}$, and each radial operator $P_{\mu_j(T)}$ is now realized on the interval $[c\lambda^{-1/2},b]$ with Dirichlet or Neumann boundary conditions at both endpoints. The scaling $x=\mu_j(T)^{-1/(\beta+2)}s$ maps this to $\mu_j(T)^{2/(\beta+2)}P_1(\beta)$ acting on $L^2(a_{j,\lambda,c},L_j)$, with
$$
a_{j,\lambda,c}=c\lambda^{-1/2}\mu_j(T)^{1/(\beta+2)},\qquad L_j=b\mu_j(T)^{1/(\beta+2)}.
$$
Denoting by $\nu_k^{D/N}(\beta,a,L)$ the $k$-th eigenvalue of $P_1(\beta)$ on $L^2(a,L)$ with Dirichlet or Neumann boundary conditions at both endpoints, the eigenvalues of the Friedrichs Laplacian on $\Omega_{T,\lambda,c}$ are
$$
\mu_j(T)^{\frac{2}{\beta+2}}\,\nu_k^{D/N}(\beta,a_{j,\lambda,c},L_j),\qquad j,k\geq 1.
$$
The model potential $V(s)=\frac{\beta n(\beta n+4)}{16s^2}+s^\beta$ tends to $+\infty$ at both endpoints $s=0$ and $s=+\infty$. Moreover, since $\beta>\beta_c=\frac{2}{n}$, we have $\frac{\beta n(\beta n+4)}{16}>\frac{3}{4}$, so $P_1(\beta)$ is in the limit-point case at $s=0$ and coincides with its Friedrichs realization on $L^2(0,+\infty)$.

We write the counting function as
$$
N_{\Omega_{T,\lambda,c}}^{D/N}(\lambda)
=
\sum_{j,k}\mathbf 1_{\big\{\mu_j(T)^{\frac{2}{\beta+2}}\nu_k^{D/N}(\beta,a_{j,\lambda,c},L_j)\leq\lambda\big\}}
+R_{0,T,\lambda,c}^{D/N}(\lambda),
$$
where the sum is over the positive tangential eigenvalues $\mu_j(T)>0$ and $k\geq 1$. Again $R_{0,T,\lambda,c}^D(\lambda)=0$, while $R_{0,T,\lambda,c}^N(\lambda)=\bigO(\lambda^{1/2})$ uniformly, hence negligible. We fix a cutoff $K\geq 1$.

Let
$$
m_0=\inf_{\beta\in[\bar\beta,\bm]}\inf_{s>0}\left(\frac{\beta n(\beta n+4)}{16s^2}+s^\beta\right)>0.
$$
Every eigenvalue $\nu_k^{D/N}(\beta,a,L)$ is bounded from below by $m_0$, uniformly in $a,L,k$ and $\beta\in[\bar\beta,\bm]$. Hence any contributing pair $(j,k)$ satisfies $\mu_j(T)^{\frac{2}{\beta+2}}m_0\leq\lambda$, and therefore
$$
a_{j,\lambda,c}=c\lambda^{-1/2}\mu_j(T)^{\frac{1}{\beta+2}}\leq c\,m_0^{-1/2}.
$$
Moreover, since $T$ is mesoscopic, $L_j=b\mu_j(T)^{\frac{1}{\beta+2}}\geq Cb\,\eta^{-\frac{2}{\beta+2}}\to+\infty$ uniformly in the positive tangential modes.

For the head $k\leq K$: after choosing $c>0$ sufficiently small (so that all the left endpoints $a_{j,\lambda,c}$ are uniformly small) and then taking $\lambda$ large, Proposition \ref{prop:1d-truncation} gives, for all $k\leq K$,
$$
\nu_k^{D/N}(\beta,a_{j,\lambda,c},L_j)=\nu_k(\beta)(1+\mathrm{O}(\varepsilon)),
$$
uniformly for all contributing pairs $(j,k)$, for $\beta\in[\bar\beta,\bm]$ and $T\in\mathscr S$. By the uniform tangential Weyl law, the contribution of these finitely many indices, normalized by $\lambda^{\frac{n}{2}+\frac{n}{4}\beta}$, equals $\gamma_n\vert T\vert\sum_{k=1}^K\nu_k(\beta)^{-d_H/2}(1+\mathrm{O}(\varepsilon)+\mathrm{o}(1))$.

For the tail $k>K$: Lemma \ref{lem:uniform-1d-Neumann} applies on the truncated intervals $(a_{j,\lambda,c},L_j)$ as well, giving $\nu_k^{D/N}(\beta,a_{j,\lambda,c},L_j)\geq C^{-1}k^{2\beta/(\beta+2)}$ uniformly. Exactly as in the finite-collar case, if a pair $(j,k)$ contributes then $\mu_j(T)\leq C^{\frac{\beta+2}{2}}\lambda^{\frac{\beta+2}{2}}k^{-\beta}$, and the uniform tangential Weyl upper bound gives a normalized tail bounded above, uniformly in $\lambda$, by $C''\vert T\vert\sum_{k>K}k^{-n\beta/2}$, which tends to $0$ as $K\to+\infty$, uniformly for $\beta\in[\bar\beta,\bm]$, because $n\bar\beta/2>1$.

Given $\varepsilon>0$, choose first $K$ large enough so that the normalized tail is bounded by $\varepsilon$, uniformly for $\beta\in[\bar\beta,\bm]$ and for all mesoscopic simplexes under consideration. This is possible because $n\bar\beta/2>1$ and the volumes of these simplexes are uniformly bounded. Then choose $c_\varepsilon\in(0,\pi/2)$ small enough so that Proposition \ref{prop:1d-truncation} gives the relative error $\mathrm{O}(\varepsilon)$ for all $k\leq K$ and all $c\in(0,c_\varepsilon)$. Letting $\lambda\to+\infty$, we obtain, for every $c\in(0,c_\varepsilon)$,
$$
N_{\Omega_{T,\lambda,c}}^{D/N}(\lambda) = A(\beta,n)\,\vert T\vert\,\lambda^{\frac{n}{2}+\frac{n}{4}\beta} \, (1+\mathrm{O}(\varepsilon)+\mathrm{o}(1))
$$
as $\lambda\to+\infty$, uniformly with respect to $\beta\in[\bar\beta,\bm]$ and $T\in\mathscr S$. This proves Lemma \ref{lem_constant-cell}.
\end{proof}

\subsection{Morse-Bott critical maxima and Gaussian integrals}\label{app:gauss} 

\begin{proposition}\label{prop:gauss}
Let $(M,g)$ be a compact Riemannian manifold of dimension $n$, let $f,a\in C^\infty(M)$ and assume that $f$ achieves its maximum $f_{\mathrm{max}}$ on a submanifold $W=f^{-1}(f_{\mathrm{max}})$ of codimension $d$. Assume that $f$ is Morse-Bott along $W$, that is, $-\mathrm{Hess}(f)$ is positive definite on the normal bundle $N_W$. Then, as $\tau\to+\infty$,
\begin{multline*}
\int_M e^{\tau f(m)}a(m)\,d\vol_g(m) \\
= e^{\tau f_{\mathrm{max}}}\Big(\frac{2\pi}{\tau}\Big)^{\frac{d}{2}} \int_W \frac{a(w)}{\sqrt{\det(-\mathrm{Hess}(f)(w)\vert_{N_wW})}}\,d\vol_{g\vert_W}(w) + \mathrm{o}\big(e^{\tau f_{\mathrm{max}}}\tau^{-\frac{d}{2}}\big).
\end{multline*}
\end{proposition}

\begin{proof}
By the Morse-Bott lemma, for each $w\in W$ there exist local coordinates $(y,z)\in\R^{n-d}\times\R^d$ such that $W=\{z=0\}$ and
$$
f(y,z)=f_{\mathrm{max}}-\frac{1}{2}\langle H(y)z,z\rangle+\mathrm{O}(\Vert z\Vert^3),
\qquad
H(y)>0.
$$
Integrating first in $z$ gives a Gaussian integral with parameter $\tau$, producing the factor $(2\pi/\tau)^{d/2}$ and the determinant. The remainder follows from standard Laplace method estimates. See for instance \cite{Bot54, Wong01}.
\end{proof}

\subsection{Boundary distance for the model metric}\label{app:boundary-distance}

In this appendix we justify Remark \ref{rem:boundary-distance} (in Section \ref{sec:normal}) in the constant exponent case and compute the constant $c_\beta$.

\begin{lemma}\label{lem:boundary-distance}
Assume that $\beta>0$ is constant and that, in a collar neighborhood $(0,a]\times M$, the metric is $g=dx^2+x^{-\beta}h_0$.
Let $\overline{X}$ be the metric completion of $((0,a]\times M,g)$ and identify $\{0\}\times M$ with the boundary copy of $M$ in $\overline{X}$. Let $d_\partial$ be the restriction of the distance of $\overline{X}$ to $\{0\}\times M$. Then for every $p\in M$, as $q\to p$ in $M$,
$$
d_\partial(p,q)=c_\beta\,d_{h_0}(p,q)^{\frac{2}{2+\beta}}(1+\mathrm{o}(1)),
$$
where
$$
c_\beta = 2\left(\frac{2+\beta}{4}\right)^{\frac{2}{2+\beta}} \left(\frac{\sqrt{\pi}}{\beta}\frac{\Gamma(\frac{1}{\beta})}{\Gamma(\frac{1}{\beta}+\frac{1}{2})}\right)^{\frac{\beta}{2+\beta}}.
$$
\end{lemma}

\begin{proof}
We proceed in two steps.

\noindent
\emph{Step 1: the flat model.}
Consider $(0,+\infty)\times\R^n$ with the metric $g_\beta=dx^2+x^{-\beta}\vert dy\vert^2$.
Let $d_\beta$ be the distance induced on $\{0\}\times\R^n$ by the metric completion. By translation and rotation invariance in $y$, $d_\beta(y_1,y_2)$ depends only on $\vert y_2-y_1\vert$. Moreover, the dilation $\delta_s(x,y)=(sx,s^{1+\frac{\beta}{2}}y)$, $s>0$, satisfies $\delta_s^*g_\beta=s^2 g_\beta$, hence it multiplies lengths by $s$ and therefore distances by $s$. Setting $\theta=\frac{2}{2+\beta}$, this scaling implies that there exists a constant $c_\beta>0$ such that $d_\beta(y_1,y_2)=c_\beta\vert y_2-y_1\vert^\theta$.
It remains to compute $c_\beta=d_\beta(0,e_1)$.

By symmetry, any minimizing curve joining $(0,0)$ to $(0,e_1)$ is contained in the $2$-plane spanned by the $x$-axis and the direction $e_1$. We reduce to $n=1$ and consider the metric $dx^2+x^{-\beta}dy^2$ on $(0,+\infty)\times\R$.
A minimizing curve has monotone $y$, so we may parametrize it by $y\in[0,1]$ as $x=x(y)$. Its length is
$$
L(x)=\int_0^1\sqrt{x'(y)^2+x(y)^{-\beta}}\,dy.
$$
Since the integrand does not depend explicitly on $y$, the quantity $\frac{x^{-\beta}}{\sqrt{x'^2+x^{-\beta}}}$ is constant along an Euler-Lagrange solution. Let $x_{\max}$ be the maximum of $x$ along the minimizer; by symmetry it is attained at $y=\frac{1}{2}$, where $x'=0$. Hence the constant equals $x_{\max}^{-\beta/2}$, and on $[0,\frac{1}{2}]$ we have
$$
x'^2=x^{-2\beta}\big(x_{\max}^\beta-x^\beta\big),
\qquad
\frac{dy}{dx}=x^\beta x_{\max}^{-\beta/2}\frac{1}{\sqrt{1-(x/x_{\max})^\beta}}.
$$
Integrating from $0$ to $x_{\max}$ and setting $t=x/x_{\max}$ gives the constraint $y(x_{\max})=\frac{1}{2}$:
$$
\frac{1}{2}=\int_0^{x_{\max}}\frac{dy}{dx}\,dx
=x_{\max}^{1+\frac{\beta}{2}}\int_0^1\frac{t^\beta\,dt}{\sqrt{1-t^\beta}}
=x_{\max}^{1+\frac{\beta}{2}}I_1,
\qquad I_1=\int_0^1\frac{t^\beta\,dt}{\sqrt{1-t^\beta}}.
$$
Next, using $\sqrt{x'^2+x^{-\beta}}=x^{-\beta}x_{\max}^{\beta/2}$ and the above expression for $dy/dx$, we obtain on $[0,\frac{1}{2}]$
$$
\frac{dL}{dx}=\sqrt{x'^2+x^{-\beta}}\frac{dy}{dx}=\frac{1}{\sqrt{1-(x/x_{\max})^\beta}}.
$$
Therefore,
$$
\frac{L}{2}=\int_0^{x_{\max}}\frac{dx}{\sqrt{1-(x/x_{\max})^\beta}}
=x_{\max}\int_0^1\frac{dt}{\sqrt{1-t^\beta}}
=x_{\max}I_0,
\qquad I_0=\int_0^1\frac{dt}{\sqrt{1-t^\beta}}.
$$
Combining the two identities yields $c_\beta=L=2x_{\max}I_0=2I_0(2I_1)^{-\frac{2}{2+\beta}}$.
Finally, we claim that $I_1=\frac{2}{2+\beta}I_0$ by using the change of variable $s=t^\beta$:
$$
I_0=\frac{1}{\beta}\int_0^1 s^{\frac{1}{\beta}-1}(1-s)^{-1/2}\,ds,
\qquad
I_1=\frac{1}{\beta}\int_0^1 s^{\frac{1}{\beta}}(1-s)^{-1/2}\,ds
=\frac{\frac{1}{\beta}}{\frac{1}{\beta}+\frac{1}{2}}I_0
=\frac{2}{2+\beta}I_0.
$$
Hence
$$
c_\beta=2I_0\left(\frac{4}{2+\beta}I_0\right)^{-\frac{2}{2+\beta}}
=2\left(\frac{2+\beta}{4}\right)^{\frac{2}{2+\beta}}I_0^{\frac{\beta}{2+\beta}}.
$$
Now, using the Beta function, we have $I_0=\frac{1}{\beta}B\big(\frac{1}{\beta},\frac{1}{2}\big)$, and since $B(x,y)=\frac{\Gamma(x)\Gamma(y)}{\Gamma(x+y)}$ and $\Gamma(\frac{1}{2})=\sqrt{\pi}$, we obtain $
I_0=\frac{\sqrt{\pi}}{\beta}\frac{\Gamma(\frac{1}{\beta})}{\Gamma(\frac{1}{\beta}+\frac{1}{2})}$, and the stated formula for $c_\beta$.

\noindent
\emph{Step 2: localization on $(0,a]\times M$.}
Fix $p\in M$ and choose $h_0$-normal coordinates $z$ near $p$ so that $h_0(p)$ is Euclidean at $z=0$. Let $q\to p$ and set $d=d_{h_0}(p,q)$ and $\theta=\frac{2}{2+\beta}$. In these coordinates, the $h_0$-coordinates of $q$ satisfy $\vert z(q)\vert=d(1+\mathrm{o}(1))$. Set $u=z(q)/d$, so that $\vert u\vert=1+\mathrm{o}(1)$.

Define the scaling map $S_d(s,z)=(x,y)=(d^\theta s,dz)$. The rescaled metric $\tilde{g}_d=d^{-2\theta}S_d^*g$ satisfies $\tilde{g}_d=ds^2+s^{-\beta}h_0(dz)$.
Since $h_0$ is continuous, on every fixed cylinder $\{0\leq s\leq R\ \mid\ \vert z\vert\leq R\}$ one has
$$
(1-\omega(dR))(ds^2+s^{-\beta}\vert dz\vert^2)\leq \tilde{g}_d\leq (1+\omega(dR))(ds^2+s^{-\beta}\vert dz\vert^2),
$$
with $\omega(r)\to 0$ as $r\to 0$.

Moreover, an explicit three-piece path (go up to $x=r$, move tangentially, go down) shows that $d_\partial(p,q)\leq 3d^\theta$. Hence $d_{\tilde{g}_d}\big((0,0),(0,u)\big)=d^{-\theta}d_\partial(p,q)\leq 3$, and any minimizing curve for $d_{\tilde{g}_d}((0,0),(0,u))$ stays in a fixed cylinder $\{0\leq s\leq R,\ \vert z\vert\leq R\}$ with $R$ independent of $d$.

Therefore the distance in $(\{0\}\times\R^n,\tilde{g}_d)$ between $(0,0)$ and $(0,u)$ converges to the flat boundary distance $d_\beta(0,u)$ as $d\to 0$. Using Step 1, we obtain
$$
d_{\tilde{g}_d}\big((0,0),(0,u)\big)=c_\beta\vert u\vert^\theta(1+\mathrm{o}(1))
=c_\beta(1+\mathrm{o}(1)).
$$
Multiplying by $d^\theta$ gives $d_\partial(p,q)=c_\beta d^\theta(1+\mathrm{o}(1))$, which proves the lemma.
\end{proof}

\begin{remark}\label{rem:boundary-distance-variable}
Assume that in a collar neighborhood the metric has the form $g=dx^2+x^{-\beta(m)}h_0(m)$ with $\beta\in C^1(M)$ and $h_0\in C^1(M)$.
Fix $p\in M$. Then, as $q\to p$ in $M$,
$$
d_\partial(p,q)=c_{\beta(p)}\,d_{h_0}(p,q)^{\frac{2}{2+\beta(p)}}(1+\mathrm{o}(1)),
$$
where $c_{\beta(p)}$ is given by Lemma \ref{lem:boundary-distance} with $\beta$ replaced by $\beta(p)$.
Indeed, in Step 2 of the above proof the same rescaling gives a family of metrics $\tilde g_d$ that converges uniformly on fixed cylinders to the flat model metric $ds^2+s^{-\beta(p)}\vert dz\vert^2$ since $\beta(dz)=\beta(p)+\mathrm{O}(d)$ and $h_0(dz)=h_0(p)+\mathrm{O}(d)$.
\end{remark}


\end{document}